\documentclass[12pt,letterpaper,final,twoside,leqno]{amsart}
\usepackage{amsmath,amssymb,amsthm,amsfonts,amscd,amsopn}
\usepackage{eucal,mathrsfs}
\usepackage{xspace,chngcntr}
\usepackage{rotating,mathtools}
\usepackage[all,cmtip,color]{xy}\xyoption{dvips}

\usepackage{comment} 
\usepackage{supertabular}
\usepackage
{hyperref}
\usepackage[dvipsnames,usenames]{xcolor}
\hypersetup{
    colorlinks=true,
    linkcolor={blue!50!black},
    citecolor={green!50!black},
    urlcolor={red!80!black}
}

\usepackage{datetime,chngcntr}
\usepackage{paralist}

\usepackage{enumitem}
\usepackage{eurosym}

\usepackage{mfirstuc}

   [2017/06/13 v0.1 finetuning mabliautoref and theorem setup]

\RequirePackage{hyperref}

\newcommand{\sectionsize}{\footnotesize}
\newcommand{\theoremsize}{\small}

\renewcommand{\subsectionautorefname}{\sectionsize\sf \subsectionautorefname}

\makeatletter
\@ifdefinable\equationname{\let\equationname\equationautorefname}
\def\equationautorefname~#1\@empty\@empty\null{\protect{\theoremsize\sf
    (#1\@empty\@empty\null)}}%
\@ifdefinable\AMSname{\let\AMSname\AMSautorefname}
\def\AMSautorefname~#1\@empty\@empty\null{\rm (#1\@empty\@empty\null)}%
\@ifdefinable\itemname{\let\itemname\itemautorefname}
\def\itemautorefname~#1\@empty\@empty\null{\theoremsize\sf #1\@empty\@empty\null%
}%
\makeatother

\RequirePackage{amsthm}
\RequirePackage{aliascnt}
\newcommand{\basetheorem}[3]{%
    \newtheorem{#1}{#2}[#3]
    \newtheorem*{#1*}{#2}
    \expandafter\def\csname #1autorefname\endcsname{#2}
}%
\newcommand{\maketheorem}[3]{%
    \newaliascnt{#1}{#2}
    \newtheorem{#1}[#1]{\theoremsize\sf #3}
    \aliascntresetthe{#1}
    \expandafter\def\csname #1autorefname\endcsname{\theoremsize\sf #3}
    \newtheorem{#1*}{#3}
}%
\newcommand{\baseremark}[3]{%
    \newtheorem{#1}{#2}{#3}
    \newtheorem*{#1*}{#2}
    \expandafter\def\csname #1autorefname\endcsname{#2}
}%
\newcommand{\makeremark}[3]{%
    \newaliascnt{#1}{#2}
    \newtheorem{#1}[equation]{#3}
    \aliascntresetthe{#1}
    \expandafter\def\csname #1autorefname\endcsname{\theoremsize\sf #3}
    \newtheorem{#1*}{#3}
}%
\theoremstyle{plain}   

\basetheorem{theorem}{Theorem}{section}

\newcommand{\mypagesize}{
\textwidth= 6.5in
\textheight=8.75in
\voffset-.5in
\hoffset-.75in
\marginparwidth=56pt
\footskip.5in
}

\newcounter{are-there-sections}
\newcommand{\zzzzz}{}

\mypagesize
\usepackage{amsbsy}

\usepackage{marvosym}

\DeclareMathAlphabet{\smallchanc}{OT1}{pzc}%
                                 {m}{it}
\DeclareFontFamily{OT1}{pzc}{}
\DeclareFontShape{OT1}{pzc}{m}{it}%
             {<-> s * [1.100] pzcmi7t}{}
\DeclareMathAlphabet{\mathchanc}{OT1}{pzc}%
                                 {m}{it}

\newcommand{\mcL}{\mathchanc{L}}

\newcommand{\mcR}{\mathchanc{R}}

\newcommand{\mcX}{\mathchanc{X}}

\newcommand{\mcx}{\mathchanc{x}}

\newcommand{\sI}{\mathscr{I}}
\newcommand{\sJ}{\mathscr{J}}

\newcommand{\sL}{\mathscr{L}}

\newcommand{\sO}{\mathscr{O}}

\newcommand{\sfA}{{\sf A}}

\newcommand{\sfM}{{\sf M}}

\newcommand{\sfP}{{\sf P}}

\newcommand{\sfh}{{\sf h}}

\newcommand{\bA}{\mathbb{A}}

\newcommand{\bC}{\mathbb{C}}

\newcommand{\bH}{\mathbb{H}}

\newcommand{\bN}{\mathbb{N}}

\newcommand{\bP}{\mathbb{P}}
\newcommand{\bQ}{\mathbb{Q}}

\newcommand{\bZ}{\mathbb{Z}}

\def\frm{\mathfrak{m}}
\newcommand{\frn}{\mathfrak{n}}

\DeclareSymbolFont{largesymbolsA}{U}{jkpexa}{m}{n}
\SetSymbolFont{largesymbolsA}{bold}{U}{jkpexa}{bx}{n}
\DeclareMathSymbol{\varprod}{\mathop}{largesymbolsA}{16}
\makeatletter
\newcommand{\LeftEqNo}{\let\veqno\@@leqno}
\makeatother
\newcommand{\into}{\hookrightarrow}

\newcommand{\onto}{\twoheadrightarrow}

\newcommand{\properideal}%
        {\subsetneq}

\newcommand{\wt}{\widetilde}
\newcommand{\what}{\widehat}

\newcommand{\leteq}{\colon\!\!\!=}

\newcommand\dash[1]{\rule[-.2ex]{#1}{.4pt}}

\DeclareMathOperator{\ann}{\mathchanc{ann}}

\DeclareMathOperator{\kar}{char}

\newcommand{\sExt}[0]{{\mathchanc{Ext}}}

\DeclareMathOperator{\Hom}{Hom}
\newcommand{\sHom}[0]{{\mathchanc{Hom}}}

\DeclareMathOperator{\im}{{im}}

\DeclareMathOperator{\invlim}{\ensuremath{\underset{\longleftarrow}{\lim}}}

\DeclareMathOperator{\length}{{length}}

\newcommand{\lotimes}{\overset{L}{\otimes}}

\newcommand{\red}{\mathrm{red}}

\DeclareMathOperator{\Spec}{{Spec}}

\DeclareMathOperator{\Soc}{{Soc}}
\DeclareMathOperator{\supp}{{supp}}

\DeclareMathOperator{\skvert}{{\,\vert\,}}

\newcommand{\factor}[2]{\left. \raise .2em\hbox{\ensuremath{#1}\vphantom{$I^d$}}
\hskip -.1em \right/ \hskip -.4em \raise -.3em\hbox{\ensuremath{#2}}}%
\newcommand\mtimes[3]{{\varprod_{#1}^{#2}}_{\raise 1ex \hbox{\scriptsize #3}}}%

\newcommand{\myR}{{\mcR\!}}
\newcommand{\myL}{{\mcL\!}}

\newcommand{\blank}{\dash{1em}}

\newcommand{\kdot}{{{\,\begin{picture}(1,1)(-1,-2)\circle*{2}\end{picture}\,}}}

\newcommand{\mydot}{\kdot}

\newcommand{\cmx}[1]{{#1}^{\raisebox{.15em}{\ensuremath\kdot}}}

\newcommand{\uldcx}[1]{\underline{\omega}^\kdot_{#1}}
\newcommand{\dcx}[1]{{\omega}^\kdot_{#1}}

\def\dimcoh#1.#2.#3.{h^{#1}(#2,#3)}
\def\hypcoh#1.#2.#3.{\mathbb H_{\vphantom{l}}^{#1}(#2,#3)}
\def\loccoh#1.#2.#3.#4.{H^{#1}_{#2}(#3,#4)}
\def\dimloccoh#1.#2.#3.#4.{h^{#1}_{#2}(#3,#4)}
\def\lochypcoh#1.#2.#3.#4.{\mathbb H^{#1}_{#2}(#3,#4)}
\def\seslong#1.#2.#3.{0  \longrightarrow  #1   \longrightarrow 
 #2 \longrightarrow #3 \longrightarrow 0} 
\def\sesshort#1.#2.#3.{0
 \rightarrow #1 \rightarrow #2 \rightarrow #3 \rightarrow 0}
\def\dist#1.#2.#3.{  #1   \longrightarrow 
 #2 \longrightarrow #3 \stackrel{+1}{\longrightarrow} } 
\def\CDdist#1.#2.#3.{  #1   @>>>  #2  @>>>   #3 @>+1>> }  
\def\shortses#1.#2.#3.{0  \rightarrow  #1   \rightarrow 
 #2  \rightarrow   #3 \rightarrow  0}
\def\shortdist#1.#2.#3.{  #1   \rightarrow 
 #2  \rightarrow   #3 \stackrel{+1}{\rightarrow} }  
\def\ddist#1.#2.#3.#4.#5.#6.{\CD
#1 @>>> #2 @>>> #3 @>+1>> \\
@VVV @VVV @VVV \\
#4 @>>> #5 @>>> #6 @>+1>> 
\endCD}
\def\ddistun#1.#2.#3.#4.#5.#6.{\CD
#1 @>>> #2 @>>> #3 @>+1>> \\
@. @VVV @VVV  \\
#4 @>>> #5 @>>> #6 @>+1>> 
\endCD}
\def\Iff#1#2#3{
\hfil\hbox{\hsize =#1
\vtop{\noin #2}
\hskip.5cm 
\lower.5\baselineskip\hbox{$\Leftrightarrow$}\hskip.5cm
\vtop{\noin #3}}\hfil\medskip}
\newcommand{\union}\cup
\newcommand{\intersect}\cap
\newcommand{\Union}\bigcup
\newcommand{\Intersect}\bigcap
\def\myoplus#1.#2.{\underset #1 \to {\overset #2 \to \oplus}}

\newcommand{\resto}[1]{\raise -.5ex\hbox{$\vert$}_{#1}}

\newcommand{\ses}{short exact sequence\xspace}

\newcommand\dbcx[1]{\underline\Omega_{#1}^0}

\newcommand{\DB}{{Du\thinspace\nolinebreak Bois}\xspace}
\newcommand{\sings}{singularities\xspace}

\begin{document}
\makeatletter
\definecolor{brick}{RGB}{204,0,0}
\def\@cite#1#2{{%
 \m@th\upshape\mdseries[{\small\sffamily #1}{\if@tempswa, \small\sffamily
   \color{brick} #2\fi}]}}
\newcommand{\sandor}{{\color{blue}{S\'andor \mdyydate\today}}}
\newenvironment{refmr}{}{}
\newcommand{\biblio}{%
\bibliographystyle{/home/kovacs/tex/TeX_input/skalpha} 
\bibliography{/home/kovacs/tex/TeX_input/Ref} 
}
%
\setitemize[1]{leftmargin=*,parsep=0em,itemsep=0.125em,topsep=0.125em}
\newcommand\james{M\hskip-.1ex\raise .575ex \hbox{\text{c}}\hskip-.075ex Kernan}

%
\renewcommand\thesubsection{\thesection.\Alph{subsection}}
\renewcommand\subsection{
  \renewcommand{\sfdefault}{phv}
  \@startsection{subsection}%
  {2}{0pt}{-\baselineskip}{.2\baselineskip}{\raggedright
    \sffamily\itshape\small
  }}
\renewcommand\section{
  \renewcommand{\sfdefault}{phv}
  \@startsection{section} %
  {1}{0pt}{\baselineskip}{.2\baselineskip}{\centering
    \sffamily
    \scshape
}}

\setlist[enumerate]{itemsep=3pt,topsep=3pt,leftmargin=2em,label={\rm (\roman*)}}
\newlist{enumalpha}{enumerate}{1}
\setlist[enumalpha]{itemsep=3pt,topsep=3pt,leftmargin=2em,label=(\alph*\/)}
\newcounter{parentthmnumber}
\setcounter{parentthmnumber}{0}
\newcounter{currentparentthmnumber}
\setcounter{currentparentthmnumber}{0}
\newcounter{nexttag}
\newcommand{\setnexttag}{%
  \setcounter{nexttag}{\value{enumi}}%
  \addtocounter{nexttag}{1}%
}
\newcommand{\placenexttag}{%
\tag{\roman{nexttag}}%
}

\newenvironment{thmlista}{%
\label{parentthma}
\begin{enumerate}
}{%
\end{enumerate}
}
\newlist{thmlistaa}{enumerate}{1}
\setlist[thmlistaa]{label=(\arabic*), ref=\autoref{parentthm}\thethm(\arabic*)}
\newcommand*{\parentthmlabeldef}{%
  \expandafter\newcommand
  \csname parentthm\the\value{parentthmnumber}\endcsname
}
\newcommand*{\ptlget}[1]{%
  \romannumeral-`\x
  \ltx@ifundefined{parentthm\number#1}{%
    \ltx@space
    \parentthmundefined
  }{%
    \expandafter\ltx@space
    \csname mymacro\number#1\endcsname
  }%
}  
\newcommand*{\parentthmundefined}{\textbf{??}}
\parentthmlabeldef{parentthm}
\newenvironment{thmlistr}{%
\label{parentthm}
\begin{thmlistrr}}{%
\end{thmlistrr}}
\newlist{thmlistrr}{enumerate}{1}
\setlist[thmlistrr]{label=(\roman*), ref=\autoref{parentthm}(\roman*)}
\newcounter{proofstep}%
\setcounter{proofstep}{0}%
\newcommand{\pstep}[1]{%
  \smallskip
  \noindent
  \emph{{\sc Step \arabic{proofstep}:} #1.}\addtocounter{proofstep}{1}}
\newcounter{lastyear}\setcounter{lastyear}{\the\year}
\addtocounter{lastyear}{-1}
\newcommand\sideremark[1]{%
\normalmarginpar
\marginpar
[
\hskip .45in
\begin{minipage}{.75in}
\tiny #1
\end{minipage}
]
{
\hskip -.075in
\begin{minipage}{.75in}
\tiny #1
\end{minipage}
}}
\newcommand\rsideremark[1]{
\reversemarginpar
\marginpar
[
\hskip .45in
\begin{minipage}{.75in}
\tiny #1
\end{minipage}
]
{
\hskip -.075in
\begin{minipage}{.75in}
\tiny #1
\end{minipage}
}}
\newcommand\Index[1]{{#1}\index{#1}}
\newcommand\inddef[1]{\emph{#1}\index{#1}}
\newcommand\noin{\noindent}
\newcommand\hugeskip{\bigskip\bigskip\bigskip}
\newcommand\smc{\sc}
\newcommand\dsize{\displaystyle}
\newcommand\sh{\subheading}
\newcommand\nl{\newline}
\newcommand\get{\input /home/kovacs/tex/latex/} 
\newcommand\Get{\Input /home/kovacs/tex/latex/} 
\newcommand\toappear{\rm (to appear)}
\newcommand\mycite[1]{[#1]}
\newcommand\myref[1]{(\ref{#1})}
\newcommand{\parref}[1]{\eqref{\bf #1}}
\newcommand\myli{\hfill\newline\smallskip\noindent{$\bullet$}\quad}
\newcommand\vol[1]{{\bf #1}\ } 
\newcommand\yr[1]{\rm (#1)\ } 
\newcommand\cf{cf.\ \cite}
\newcommand\mycf{cf.\ \mycite}
\newcommand\te{there exist\xspace}
\newcommand\st{such that\xspace}
\newcommand\CM{Cohen-Macaulay\xspace}
\newcommand\myskip{3pt}
\newtheoremstyle{bozont}{3pt}{3pt}%
     {\itshape}
     {}
     {\bfseries}
     {.}
     {.5em}
     {\thmname{#1}\thmnumber{ #2}\thmnote{ #3}}%
\newtheoremstyle{bozont-sub}{3pt}{3pt}%
     {\itshape}
     {}
     {\bfseries}
     {.}
     {.5em}
     {\thmname{#1}\ \arabic{section}.\arabic{thm}.\thmnumber{#2}\thmnote{ \rm #3}}
\newtheoremstyle{bozont-named-thm}{3pt}{3pt}%
     {\itshape}
     {}
     {\bfseries}
     {.}
     {.5em}
     {\thmname{#1}\thmnumber{#2}\thmnote{ #3}}
\newtheoremstyle{bozont-named-bf}{3pt}{3pt}%
     {}
     {}
     {\bfseries}
     {.}
     {.5em}
     {\thmname{#1}\thmnumber{#2}\thmnote{ #3}}
\newtheoremstyle{bozont-named-sf}{3pt}{3pt}%
     {}
     {}
     {\sffamily}
     {.}
     {.5em}
     {\thmname{#1}\thmnumber{#2}\thmnote{ #3}}
\newtheoremstyle{bozont-named-sc}{3pt}{3pt}%
     {}
     {}
     {\scshape}
     {.}
     {.5em}
     {\thmname{#1}\thmnumber{#2}\thmnote{ #3}}
\newtheoremstyle{bozont-named-it}{3pt}{3pt}%
     {}
     {}
     {\itshape}
     {.}
     {.5em}
     {\thmname{#1}\thmnumber{#2}\thmnote{ #3}}
\newtheoremstyle{bozont-sf}{3pt}{3pt}%
     {}
     {}
     {\sffamily}
     {.}
     {.5em}
     {\thmname{#1}\thmnumber{ #2}\thmnote{ \rm #3}}
\newtheoremstyle{bozont-sc}{3pt}{3pt}%
     {}
     {}
     {\scshape}
     {.}
     {.5em}
     {\thmname{#1}\thmnumber{ #2}\thmnote{ \rm #3}}
\newtheoremstyle{bozont-remark}{3pt}{3pt}%
     {}
     {}
     {\scshape}
     {.}
     {.5em}
     {\thmname{#1}\thmnumber{ #2}\thmnote{ \rm #3}}
\newtheoremstyle{bozont-subremark}{3pt}{3pt}%
     {}
     {}
     {\scshape}
     {.}
     {.5em}
     {\thmname{#1}\ \arabic{section}.\arabic{thm}.\thmnumber{#2}\thmnote{ \rm #3}}
\newtheoremstyle{bozont-def}{3pt}{3pt}%
     {}
     {}
     {\bfseries}
     {.}
     {.5em}
     {\thmname{#1}\thmnumber{ #2}\thmnote{ \rm #3}}
\newtheoremstyle{bozont-reverse}{3pt}{3pt}%
     {\itshape}
     {}
     {\bfseries}
     {.}
     {.5em}
     {\thmnumber{#2.}\thmname{ #1}\thmnote{ \rm #3}}
\newtheoremstyle{bozont-reverse-sc}{3pt}{3pt}%
     {\itshape}
     {}
     {\scshape}
     {.}
     {.5em}
     {\thmnumber{#2.}\thmname{ #1}\thmnote{ \rm #3}}
\newtheoremstyle{bozont-reverse-sf}{3pt}{3pt}%
     {\itshape}
     {}
     {\sffamily}
     {.}
     {.5em}
     {\thmnumber{#2.}\thmname{ #1}\thmnote{ \rm #3}}
\newtheoremstyle{bozont-remark-reverse}{3pt}{3pt}%
     {}
     {}
     {\sc}
     {.}
     {.5em}
     {\thmnumber{#2.}\thmname{ #1}\thmnote{ \rm #3}}
\newtheoremstyle{bozont-def-reverse}{3pt}{3pt}%
     {}
     {}
     {\bfseries}
     {.}
     {.5em}
     {\thmnumber{#2.}\thmname{ #1}\thmnote{ \rm #3}}
\newtheoremstyle{bozont-def-newnum-reverse}{3pt}{3pt}%
     {}
     {}
     {\bfseries}
     {}
     {.5em}
     {\thmnumber{#2.}\thmname{ #1}\thmnote{ \rm #3}}
\newtheoremstyle{bozont-def-newnum-reverse-plain}{3pt}{3pt}%
   {}
   {}
   {}
   {}
   {.5em}
   {\thmnumber{\!(#2)}\thmname{ #1}\thmnote{ \rm #3}}
\newtheoremstyle{bozont-number}{3pt}{3pt}%
   {}
   {}
   {}
   {}
   {0pt}
   {\thmnumber{\!(#2)} }
\newtheoremstyle{bozont-step}{3pt}{3pt}%
   {\itshape}
   {}
   {\scshape}
   {}
   {.5em}
   {$\boxed{\text{\sc \thmname{#1}~\thmnumber{#2}:\!}}$}
\theoremstyle{bozont}    
\ifnum \value{are-there-sections}=0 {%
  \basetheorem{proclaim}{Theorem}{}
} 
\else {%
  \basetheorem{proclaim}{Theorem}{section}
} 
\fi
\maketheorem{thm}{proclaim}{Theorem}
\maketheorem{mainthm}{proclaim}{Main Theorem}
\maketheorem{cor}{proclaim}{Corollary} 
\maketheorem{cors}{proclaim}{Corollaries} 
\maketheorem{lem}{proclaim}{Lemma} 
\maketheorem{prop}{proclaim}{Proposition} 
\maketheorem{conj}{proclaim}{Conjecture}
\basetheorem{subproclaim}{Theorem}{proclaim}
\maketheorem{sublemma}{subproclaim}{Lemma}
\newenvironment{sublem}{%
\setcounter{sublemma}{\value{equation}}
\begin{sublemma}}
{\end{sublemma}}
\theoremstyle{bozont-sub}
\maketheorem{subthm}{equation}{Theorem}
\maketheorem{subcor}{equation}{Corollary} 
\maketheorem{subprop}{equation}{Proposition} 
\maketheorem{subconj}{equation}{Conjecture}
\theoremstyle{bozont-named-thm}
\maketheorem{namedthm}{proclaim}{}
\theoremstyle{bozont-sc}
\newtheorem{proclaim-special}[proclaim]{\specialthmname}
\newenvironment{proclaimspecial}[1]
     {\def\specialthmname{#1}\begin{proclaim-special}}
     {\end{proclaim-special}}
\theoremstyle{bozont-subremark}
\maketheorem{subrem}{equation}{Remark}
\maketheorem{subnotation}{equation}{Notation} 
\maketheorem{subassume}{equation}{Assumptions} 
\maketheorem{subobs}{equation}{Observation} 
\maketheorem{subexample}{equation}{Example} 
\maketheorem{subex}{equation}{Exercise} 
\maketheorem{inclaim}{equation}{Claim} 
\maketheorem{subquestion}{equation}{Question}
\theoremstyle{bozont-remark}
\basetheorem{subremark}{Remark}{proclaim}
\makeremark{subclaim}{subremark}{Claim}
\maketheorem{rem}{proclaim}{Remark}
\maketheorem{claim}{proclaim}{Claim} 
\maketheorem{notation}{proclaim}{Notation} 
\maketheorem{assume}{proclaim}{Assumptions} 
\maketheorem{obs}{proclaim}{Observation} 
\maketheorem{example}{proclaim}{Example} 
\maketheorem{examples}{proclaim}{Examples} 
\maketheorem{complem}{equation}{Complement}
\maketheorem{const}{proclaim}{Construction}   
\maketheorem{ex}{proclaim}{Exercise} 
\newtheorem{case}{Case} 
\newtheorem{subcase}{Subcase}   
\newtheorem{step}{Step}
\newtheorem{approach}{Approach}
\maketheorem{Fact}{proclaim}{Fact}
\newtheorem{fact}{Fact}
\newtheorem*{SubHeading*}{\SubHeadingName}%
\newtheorem{SubHeading}[proclaim]{\SubHeadingName}
\newtheorem{sSubHeading}[equation]{\sSubHeadingName}
\newenvironment{demo}[1] {\def\SubHeadingName{#1}\begin{SubHeading}}
  {\end{SubHeading}}%
\newenvironment{subdemo}[1]{\def\sSubHeadingName{#1}\begin{sSubHeading}}
  {\end{sSubHeading}} %
\newenvironment{demo-r}[1]{\def\SubHeadingName{#1}\begin{SubHeading-r}}
  {\end{SubHeading-r}}%
\newenvironment{demor}[1]{\def\SubHeadingName{#1}\begin{SubHeading-r}}
  {\end{SubHeading-r}}%
\newenvironment{subdemo-r}[1]{\def\sSubHeadingName{#1}\begin{sSubHeading-r}}
  {\end{sSubHeading-r}} %
\newenvironment{demo*}[1]{\def\SubHeadingName{#1}\begin{SubHeading*}}
  {\end{SubHeading*}}%
\maketheorem{defini}{proclaim}{Definition}
\maketheorem{defnot}{proclaim}{Definitions and notation}
\maketheorem{question}{proclaim}{Question}
\maketheorem{terminology}{proclaim}{Terminology}
\maketheorem{crit}{proclaim}{Criterion}
\maketheorem{pitfall}{proclaim}{Pitfall}
\maketheorem{addition}{proclaim}{Addition}
\maketheorem{principle}{proclaim}{Principle} 
\maketheorem{condition}{proclaim}{Condition}
\maketheorem{exmp}{proclaim}{Example}
\maketheorem{hint}{proclaim}{Hint}
\maketheorem{exrc}{proclaim}{Exercise}
\maketheorem{prob}{proclaim}{Problem}
\maketheorem{ques}{proclaim}{Question}    
\maketheorem{alg}{proclaim}{Algorithm}
\maketheorem{remk}{proclaim}{Remark}          
\maketheorem{note}{proclaim}{Note}            
\maketheorem{summ}{proclaim}{Summary}         
\maketheorem{notationk}{proclaim}{Notation}   
\maketheorem{warning}{proclaim}{Warning}  
\maketheorem{defn-thm}{proclaim}{Definition--Theorem}  
\maketheorem{convention}{proclaim}{Convention}  
\maketheorem{hw}{proclaim}{Homework}
\maketheorem{hws}{proclaim}{\protect{${\mathbb\star}$}Homework}
\newtheorem*{ack}{Acknowledgment}
\newtheorem*{acks}{Acknowledgments}
\theoremstyle{bozont-number}
\newtheorem{say}[proclaim]{}
\newtheorem{subsay}[equation]{}
\theoremstyle{bozont-def}    
\maketheorem{defn}{proclaim}{Definition}
\maketheorem{subdefn}{equation}{Definition}
\theoremstyle{bozont-reverse}    
\maketheorem{corr}{proclaim}{Corollary} 
\maketheorem{lemr}{proclaim}{Lemma} 
\maketheorem{propr}{proclaim}{Proposition} 
\maketheorem{conjr}{proclaim}{Conjecture}
\theoremstyle{bozont-remark-reverse}
\newtheorem{SubHeading-r}[proclaim]{\SubHeadingName}
\newtheorem{sSubHeading-r}[equation]{\sSubHeadingName}
\newtheorem{SubHeadingr}[proclaim]{\SubHeadingName}
\newtheorem{sSubHeadingr}[equation]{\sSubHeadingName}
\theoremstyle{bozont-reverse-sc}
\newtheorem{proclaimr-special}[proclaim]{\specialthmname}
\newenvironment{proclaimspecialr}[1]%
{\def\specialthmname{#1}\begin{proclaimr-special}}%
{\end{proclaimr-special}}
\theoremstyle{bozont-remark-reverse}
\maketheorem{remr}{proclaim}{Remark}
\maketheorem{subremr}{equation}{Remark}
\maketheorem{notationr}{proclaim}{Notation} 
\maketheorem{assumer}{proclaim}{Assumptions} 
\maketheorem{obsr}{proclaim}{Observation} 
\maketheorem{exampler}{proclaim}{Example} 
\maketheorem{exr}{proclaim}{Exercise} 
\maketheorem{claimr}{proclaim}{Claim} 
\maketheorem{inclaimr}{equation}{Claim} 
\maketheorem{definir}{proclaim}{Definition}
\theoremstyle{bozont-def-newnum-reverse}    
\maketheorem{newnumr}{proclaim}{}
\theoremstyle{bozont-def-newnum-reverse-plain}
\maketheorem{newnumrp}{proclaim}{}
\theoremstyle{bozont-def-reverse}    
\maketheorem{defnr}{proclaim}{Definition}
\maketheorem{questionr}{proclaim}{Question}
\newtheorem{newnumspecial}[proclaim]{\specialnewnumname}
\newenvironment{newnum}[1]{\def\specialnewnumname{#1}\begin{newnumspecial}}{\end{newnumspecial}}
\theoremstyle{bozont-step}
\newtheorem{bstep}{Step}
\newcounter{thisthm} 
\newcounter{thissection} 
\newcommand{\ilabel}[1]{%
  \newcounter{#1}%
  \setcounter{thissection}{\value{section}}%
  \setcounter{thisthm}{\value{proclaim}}%
  \label{#1}}
\newcommand{\iref}[1]{%
  (\the\value{thissection}.\the\value{thisthm}.\ref{#1})}
\newcounter{lect}
\setcounter{lect}{1}
\newcommand\resetlect{\setcounter{lect}{1}\setcounter{page}{0}}
\newcommand\lecture{\newpage\centerline{\sfbf Lecture \arabic{lect}}
\addtocounter{lect}{1}}
\newcounter{topic}
\setcounter{topic}{1}
\newenvironment{topic}
{\noindent{\sc Topic 
\arabic{topic}:\ }}{\addtocounter{topic}{1}\par}
\counterwithin{equation}{proclaim}
\counterwithin{figure}{section} 
\newcommand\equinsect{\numberwithin{equation}{section}}
\newcommand\equinthm{\numberwithin{equation}{proclaim}}
\newcommand\figinthm{\numberwithin{figure}{proclaim}}
\newcommand\figinsect{\numberwithin{figure}{section}}
\newenvironment{sequation}{%
\setcounter{equation}{\value{thm}}
\numberwithin{equation}{section}%
\begin{equation}%
}{%
\end{equation}%
\numberwithin{equation}{proclaim}%
\addtocounter{proclaim}{1}%
}
\newcommand{\num}{\arabic{section}.\arabic{proclaim}}
\newenvironment{pf}{\smallskip \noindent {\sc Proof. }}{\qed\smallskip}
\newenvironment{enumerate-p}{
  \begin{enumerate}}
  {\setcounter{equation}{\value{enumi}}\end{enumerate}}
\newenvironment{enumerate-cont}{
  \begin{enumerate}
    {\setcounter{enumi}{\value{equation}}}}
  {\setcounter{equation}{\value{enumi}}
  \end{enumerate}}
\let\lenumi\labelenumi
\newcommand{\rmlabels}{\renewcommand{\labelenumi}{\rm \lenumi}}
\newcommand{\rmlabelsoff}{\renewcommand{\labelenumi}{\lenumi}}
\newenvironment{heading}{\begin{center} \sc}{\end{center}}
\newcommand\subheading[1]{\smallskip\noindent{{\bf #1.}\ }}
\newlength{\swidth}
\setlength{\swidth}{\textwidth}
\addtolength{\swidth}{-,5\parindent}
\newenvironment{narrow}{
  \medskip\noindent\hfill\begin{minipage}{\swidth}}
  {\end{minipage}\medskip}
\newcommand\nospace{\hskip-.45ex}
\newcommand{\sfbf}{\sffamily\bfseries}
\newcommand{\sfbfs}{\sffamily\bfseries\small}
\newcommand{\twidle}{\textasciitilde}
\makeatother

\newcounter{stepp}
\setcounter{stepp}{0}
\newcommand{\nextstep}[1]{%
  \addtocounter{stepp}{1}%
  \begin{bstep}%
    {#1}
  \end{bstep}%
  \noindent%
}
\newcommand{\resetsteps}{\setcounter{stepp}{0}}

\title
{Deformations of log canonical and $F$-pure singularities}
\author{J\'anos Koll\'ar and S\'andor J Kov\'acs}
\date{\today}
\thanks{J\'anos Koll\'ar was supported in part by NSF Grant DMS-1362960. \\ 
\indent S\'andor Kov\'acs was supported in part by NSF Grant DMS-1565352 
and the Craig McKibben and Sarah Merner Endowed Professorship in Mathematics.}
\address{JK: Department of Mathematics, Princeton University, Fine
  Hall, Washington Road, Princeton, NJ 08544-1000, USA} 
\email{kollar@math.princeton.edu}
\urladdr{http://www.math.princeton.edu/$\sim$kollar\xspace}
\address{SK: University of Washington, Department of Mathematics, Box 354350,
Seattle, WA 98195-4350, USA}
\email{skovacs@uw.edu\xspace}
\urladdr{http://www.math.washington.edu/$\sim$kovacs\xspace}

\begin{abstract}
  We introduce a lifting property for local cohomology, which leads to a unified
  treatment of the dualizing complex for flat morphisms with semi-log-canonical,
  Du~Bois or $F$-pure fibers.  As a consequence we obtain that, in all 3 cases, the
  cohomology sheaves of the relative dualizing complex are flat and commute with base
  change.  We also derive several consequences for deformations of
  semi-log-canonical, Du~Bois and $F$-pure singularities.
\end{abstract}

\maketitle
\setcounter{tocdepth}{1}
\tableofcontents

\newcommand{\toploccohs}{liftable local cohomology\xspace}
\newcommand{\Fp}{$F$-pure\xspace}
\newcommand{\Fan}{$F$-anti-nilpotent\xspace}

\section{Introduction}\label{sec:introduction}
\noindent
One of the difficulties of higher dimensional birational geometry and moduli theory
is that the occuring  singularities are  frequently not Cohen-Macaulay.

On a proper Cohen-Macaulay scheme we have a dualizing \emph{sheaf} $\omega_X$ and
Serre duality. By contrast, on an arbitrary proper scheme we have a dualizing
\emph{complex} $\dcx X$ and the isomorphism of Serre duality is replaced by a
spectral sequence of Grothendieck duality.

The ``most important'' cohomology sheaf of the dualizing complex $\dcx X$ is
\[
\sfh^{-\dim X}(\dcx X)\simeq \omega_X,
\]
and $X$ is Cohen-Macaulay if and only if the other cohomology sheaves
$\sfh^{-i}(\dcx X)$ are all zero, cf.\ \cite[3.5.1]{Conrad00}. Thus these
$\sfh^{-i}(\dcx X)$ measure ``how far'' $X$ is from being Cohen-Macaulay; see
\autoref{prop:S_n-via-hi} for a more precise claim.  Our main result implies that in
flat families $X\to B$ with log canonical or $F$-pure fibers, the cohomology sheaves
$\sfh^{-i}(\dcx{X/B})$ are flat over $B$ and commute with base change. In particular,
being \CM is a deformation invariant property for such singularities. Note that for
flat families \CM is always an open condition but usually not a closed one.

One of the puzzles of higher dimensional singularity theory is that while the
singularities of the Minimal Model Program (log terminal, log canonical, Du~Bois,
etc.) and of positive characteristic commutative algebra ($F$-pure, $F$-injective,
etc.) are very closely related, the methods to study them are completely
different. Here we isolate the following quite powerful common property for some of
these classes.

\begin{defini}\label{def:liftable-cohom}
  Let $A$ be a noetherian ring, and $(T,\mathfrak n)$ a noetherian local
  $A$-algebra. We say that $T$ has \emph{\toploccohs over $A$} if for any noetherian
  local $A$-algebra $(R,\frm)$ and nilpotent ideal $I\subset R$ such that
  $R/I\simeq T$, the natural morphism on local cohomology
  \[
  \xymatrix{%
    H^i_\frm(R) \ar@{->>}[r] & H^i_{\mathfrak n}(T) }
  \]
  is surjective for all $i$.  
  
  We say that $T$ has \emph{\toploccohs} if it has \toploccohs over $\bZ$.
\end{defini}

\begin{rem}\label{rem:inheriting-top-loc-cohs}
  Notice that, using the above notation, if $\phi:A'\to A$ is a ring homomorphism
  from another noetherian ring $A'$ then if $T$ has \toploccohs over $A'$, then it
  also has \toploccohs over $A$. In particular, if $T$ has \toploccohs over $\bZ$,
  then it has \toploccohs over any noetherian ring $A$ justifying the above
  terminology.

  Furthermore, if $A=k$ is a field of characteristic $0$ then the notions of having
  \toploccohs over $k$ and over $\bZ$ are equivalent. This follows in one direction
  by the above and in the other by the Cohen structure theorem
  \cite[\href{https://stacks.math.columbia.edu/tag/032A}{Tag 032A}]{stacks-project}.
\end{rem}

\begin{rem}
  A closely related notion, a ring being \emph{cohomologically full}, is defined in
  \cite{DSM18}. This notion and results of this article are used in
  \cite{ConcaVarbaro18} to settle a conjecture by Herzog on ideals with square-free
  initial ideals.
\end{rem}

We prove in \autoref{DB.toploccohs} 
that \DB singularities have \toploccohs.  On the Frobenius side, the right concept
seems to be $F$-anti-nilpotent singularities, a notion introduced in
\cite{MR2460693}, that lies between $F$-pure and $F$-injective
\cite{MR3271179,ma_quy_2017}. We are very grateful to L.~Ma and K.~Schwede for
pointing out that, by a result of Ma--Schwede--Shimomoto \cite{MSS17},
$F$-anti-nilpotent singularities also have \toploccohs over their ground field. We
discuss this in \autoref{prop:F-anti-nilp-has-liftable-cohs}.





With this definition (cf.~\autoref{def:top-loc-cohs-for-schemes}), our main technical
theorem is the following.

\begin{thm}[=~\autoref{thm:main-strong}]\label{thm:main.new}
  Let $f:X\to B$ be a flat morphism of schemes that is essentially of finite type and
  let $b\in B$ 
  such that $X_{b}$ has \toploccohs over $B$. Then there exists an open neighborhood
  $X_b\subset U\subset X$ such that $\sfh^{-i}(\dcx{U/B})$ is flat over $B$ and
  commutes with base change for each $i\in\bZ$.
\end{thm}

\noin
For applications the following consequences are especially important.

\begin{cor}\label{cor:main-db-implies-slc}
  Let $f:X\to B$ be a flat morphism of schemes, essentially of finite type over
  a field $k$. Let $b\in B$ be a point. Assume that
  \begin{enumerate}
  \item\label{item:7} either $\operatorname{char} k=0$ and $X_{b}$ is \DB, e.g.,
    semi-log-canonical,
  \item\label{item:8} or $\operatorname{char} k>0$ and $X_{b}$ is \Fan, e.g., \Fp.
  \end{enumerate}
  Then there exists an open neighborhood $X_b\subset U\subset X$ such that
  $\sfh^{-i}(\dcx{U/B})$ is flat over $B$ and commutes with base change for each
  $i\in \bZ$.
\end{cor}


\autoref{cor:main-db-implies-slc}\autoref{item:7} can be viewed as a generalization
of the following \autoref{cor:cm-defo-implies-cm}\autoref{item:19}, proved in
\cite{MR2629988} for projective morphisms and in \cite{MSS17} in general
(cf.~\cite{KS13}).

\begin{cor}\label{cor:cm-defo-implies-cm}
  Let $(X,x)$ be a local scheme, essentially of finite type over a field and assume
  that
  \begin{enumerate}
  \item\label{item:19} either $\operatorname{char} k=0$ and $X$ is \DB, e.g.,
    semi-log-canonical,
  \item\label{item:20} or $\operatorname{char} k>0$ and $X$ is \Fan, e.g., \Fp.
  \end{enumerate}
  If $(X,x)$ admits a flat deformation whose generic fiber is \CM then $(X,x)$ is
  also \CM.
\end{cor}

In many cases this is quite sharp, see \autoref{cone.exmp.5} and
\autoref{thm:cm-defo-of-db} for some stronger versions.  This also gives the
following immediate corollary.

\begin{cor}[(cf.~\autoref{cor:cone-over-abelian-s3})]
  \label{cor:cone-over-abelian-is-not-smoothable}
  Let $X$ be an abelian variety of dimension at least $2$ defined over a field $k$.
  If $\kar k>0$ assume that $X$ is ordinary. Then the cone over an arbitrary
  projective embedding of $X$ is not smoothable.
\end{cor}

Note that a special case of \autoref{cor:cone-over-abelian-is-not-smoothable} over
$\bC$, the non-smoothability of the projective cone over an abelian variety of
dimension at least $2$, was proved in \cite{MR522037}.  See
\autoref{cor:cone-over-abelian-s3} for a stronger version.

\begin{demor}{The organization of the paper}
  In \autoref{sec:examples} we give examples and show some applications of the main
  results. In \autoref{sec:local-cohomology-db} we prove that a \DB local scheme has
  \emph{\toploccohs}. In \autoref{sec:fp-sings} we recall a few basic notions about
  singularities defined by the behaviour of the Frobenius morphism in positive
  characteristic and recall that $F$-anti-nilpotent \sings have \toploccohs over
  their ground field. In \autoref{sec:filtrations} and \autoref{sec:db-families-over}
  we study infinitesimal deformations of schemes with \toploccohs and prove the main
  result for families over Artinian bases. In \autoref{sec:flatness-base-change} we
  prove a rather general flatness and base change criterion, see
  \autoref{thm:generalized-flat-and-base-change}, which may be of independent
  interest and derive \autoref{thm:main.new} as relatively easy consequences of this
  and the results of \autoref{sec:db-families-over}. In
  \autoref{sec:degenerations-cm-db} we prove a criterion for $S_n$ \sings in terms of
  the dualizing complex, see \autoref{prop:S_n-via-hi}, and use this and
  \autoref{thm:main.new} to prove \autoref{cor:cm-defo-implies-cm}.
\end{demor}

\begin{demor}{Dualizing complex and its relatives}\label{def-and-not}
  The \emph{(normalized) dualizing complex} of $X$ is denoted by $\omega_X^\mydot$
  and if $X$ is of pure dimension $n$ the \emph{canonical sheaf} of $X$ is defined as
  $\omega_X\leteq \sfh^{-n}(\omega_X^\mydot)$.  Note that if $X$ is not normal, then
  this is not necessarily the push-forward of the canonical sheaf from the
  non-singular locus.
  
  We will work with \emph{three} closely related, but generally different
  \emph{objects}:
  \begin{itemize}
  \item the dualizing complex; $\dcx {X}$,
  \item the canonical sheaf; $\omega_{X}= \sfh^{-n}(\omega_X^\mydot)$, and
  \item the object defined by $\uldcx X\leteq\myR\sHom_X (\dbcx X, \dcx X)$.
    \quad (See \autoref{sec:local-cohomology-db} for a description of $\dbcx X$).
  \end{itemize}
  Note that one has a natural morphism $\uldcx X\to \dcx X$ dual to
  $\eta: \sO_X\to \dbcx X$.

  For a morphism $f:X\to B$, the \emph{(normalized) relative dualizing complex} of
  $f$ will be denoted by $\dcx{X/B}$ and if $f$ has equidimensional fibers of
  dimension $n$, then the \emph{relative canonical sheaf} of $f$ is
  $\omega_{X/B}\leteq \sfh^{-n}(\dcx{X/B})$.  If $B$ consists of a single (closed)
  point, then these notions reduce to the ones discussed above. For more details on
  relative dualizing complexes see
  \cite[\href{http://stacks.math.columbia.edu/tag/0E2S}{Tag 0E2S}]{stacks-project}
\end{demor}

\begin{ack}
  We would like to thank Johan de Jong for comments and discussions from which
  we have greatly benefited and for supplying several results we needed in
\cite{stacks-project}. We would also like to thank Linquan Ma and Karl Schwede
  for pointing us to \cite[Remark~3.4]{MSS17} and numerous useful discussions about $F$-\sings. 
\end{ack}




\section{Examples}\label{sec:examples}
\subsection{Characteristic zero}
In this section we review rational, \CM, and \DB singularities of cones in
characteristic zero and demonstrate some consequences of the main results.

\begin{exmp}[\emph{Deformations of cones}]\label{cone.exmp.5}
  Let $X$ be a projective variety and $\sL$ an ample line bundle on $X$.  Assume for
  simplicity that $X$ has rational singularities.  Let
  \[
  C_a(X,\sL):=\Spec_k \bigoplus_{r=0}^{\infty} H^0(X, \sL^r)
  \]
  be the affine cone over $X$ with conormal bundle $\sL$ and vertex $v$.  Then the
  singularity $v\in C_a(X,\sL)$ is
  \begin{enumerate}[leftmargin=3.2em,label=(\ref{cone.exmp.5}.\arabic*)]
  \item\label{item:21} rational $\Leftrightarrow$ $H^i(X, \sL^r)=0$ for every
    $i>0, r\geq 0$,
  \item\label{item:22} \CM $\Leftrightarrow$ $H^i(X, \sL^r)=0$ for every
    $\dim X>i>0, r\geq 0$ and
  \item\label{item:23} \DB $\Leftrightarrow$ $H^i(X, \sL^r)=0$ for every $i>0, r> 0$;
  \end{enumerate}
  see \cite[3.11, 3.13]{SingBook} and \cite[2.5]{GK14} for proofs.

  Let $D\subset X$ be an effective divisor with rational singularities such that
  $\sL\simeq \sO_X(D)$.  Set $\sL_D:=\sL|_D$.  There is a natural morphism
  $C_a(D, \sL_D)\to C_a(X,\sL)$ which is an embedding if and only if
  $ H^0(X, \sL^r)\onto H^0(D, \sL^{r}_D)$ is surjective for every $r\geq 0$,
  equivalently, if and only if $H^1(X, \sL^r)\into H^1(X, \sL^{r+1})$ is injective
  for every $r\geq 0$. As in \cite[3.10]{SingBook}, using Serre vanishing we get that
  \begin{enumerate}[resume,leftmargin=3.2em,label=(\ref{cone.exmp.5}.\arabic*)]
  \item $C_a(D, \sL_D)$ is a Cartier divisor of $C_a(X,\sL)$ $\Leftrightarrow$
    $ H^1(X, \sL^{r})=0$ for every $r\geq 0$.
  \end{enumerate}
  If this holds then $C_a(D, \sL_D)$ has a deformation whose generic fiber is
  $X\setminus D$. So $C_a(D, \sL_D)$ is smoothable if $X\setminus D$ is smooth.

  By looking at the cohomology of the sequences
  \setcounter{equation}{4}
  \begin{equation}
    \label{eq:1}
      0\to \sL^{r-1}\to \sL^r\to \sL^r_D\to 0 
  \end{equation}
  we see that
  \begin{enumerate}[leftmargin=3.2em,label=(\ref{cone.exmp.5}.\arabic*)]
    \setcounter{enumi}{5}
    \item $C_a(D,\sL_D)$ is Du~Bois $\Leftrightarrow$
    $H^1(X, \sL^r)\onto H^1(X, \sL^{r+1})$ is surjective for every $r\geq 0$ and
    $H^i(X, \sL^r)=0$ for every $i>1, r\geq 0$.
  \end{enumerate}
  Putting all these together we see that if $C_a(D,\sL_D)$ is Du~Bois and has a flat
  deformation to $X\setminus D$ then $v\in C_a(X,\sL)$ is a rational singularity. In
  particular, $C_a(D,\sL_D)$ is \CM.

  This is actually stronger than \autoref{cor:cm-defo-implies-cm}, but here we also
  assumed that $X\setminus D$ has rational singularities.  See also \cite{KS13} for
  closely related results.
\end{exmp}

By \autoref{cor:cm-defo-implies-cm}, if a local, \DB scheme $(X,x)$ is smoothable,
then it is \CM.  The next example shows that this is close to being optimal for some
cones.

\begin{exmp} 
  Let $(S, H)$ be a polarized K3 surface and set $X:=S\times \bP^2$. Fix $a,b\geq 1$
  and set $\sL(a,b):=\pi_1^*\sO_S(aH)\otimes \pi_2^*\sO_{\bP^2}(b)$ and let
  $D(a,b)\subset X$ be a smooth member of the associated linear system.
  
  The affine cone $C_a\bigl(D(a,b), \sL(a,b)|_{D(a,b)}\bigr)$ is a hyperplane section
  of the cone $C_a\bigl(X, \sL(a,b)\bigr)$, hence smoothable.  It is not \CM since
  $H^2\bigl(D(a,b),\sO_{D(a,b)}\bigr)=1$ and also not \DB since
  $H^1\bigl(D(a,b), \sL(a,b)|_{D(a,b)}\bigr)=1$.

  However, for any $a'>a, b'>b$ the cone
  $C_a\bigl(D(a,b), \sL(a',b')|_{D(a,b)}\bigr)$ is \DB but still not \CM.  Thus the
  cones $C_a\bigl(D(a,b), \sL(a',b')|_{D(a,b)}\bigr)$ are not smoothable.

  More generally, one gets similar examples starting with any smooth variety $X$ for
  which $H^1(X, \sO_X)=0$ but $H^i(X, \sO_X)\neq 0$ for some $2\leq i\leq \dim X-2$.
\end{exmp}

\begin{exmp}[\emph{Singularities of cones I}]\label{cor:cone-over-abelian}
  Let $X$ be a smooth, projective variety such that $K_X\equiv 0$.  Kodaira vanishing
  and (\ref{cone.exmp.5}.2--3) show that $C_a(X, \sL)$ is \DB.  It is \CM iff
  $H^i(X,\sO_X)=0$ for $0<i<\dim X$.  This and \autoref{cor:cm-defo-implies-cm} imply
  that if $C_a(X, \sL)$ is smoothable then $H^i(X,\sO_X)=0$ for $0<i<\dim X$. In
  particular, if $X$ is an abelian variety then $C_a(X,\sL)$ is not smoothable.
  In \autoref{cor:cone-over-abelian-s3} we prove that if $X$ is an abelian variety,
  then $C_a(X,\sL)$ cannot be deformed even to an $S_3$ scheme.
\end{exmp}

\subsection{All characteristics}

\begin{exmp}[\emph{Singularities of cones II}]\label{exmp:sings-of-cones}
  Let us use the notation introduced in \autoref{cone.exmp.5} and first note that
  \autoref{item:21} and \autoref{item:22} remain true in all characteristics:
  \begin{enumerate}[leftmargin=3.2em,label=(\ref{exmp:sings-of-cones}.\arabic*)]
  \item $C_a(X,\sL)$ is rational $\Leftrightarrow$ $H^i(X, \sL^r)=0$ for every
    $i>0, r\geq 0$, and
  \item $C_a(X,\sL)$ is \CM $\Leftrightarrow$ $H^i(X, \sL^r)=0$ for every
    $\dim X>i>0, r\geq 0$.
  \end{enumerate}

  \noin For future reference we add a more sophisticated version of
  \autoref{item:22}:
  \begin{enumerate}[resume,leftmargin=3.2em,label=(\ref{exmp:sings-of-cones}.\arabic*)]
  \item\label{item:24} $C_a(X,\sL)$ is $S_n$ for some $n\in\bN$ $\Leftrightarrow$
    $H^i(X, \sL^r)=0$ for every $n-1>i>0, r\geq 0$.
  \end{enumerate}
  The reader may find a proof, for instance, in \cite[4.3]{MR3123642},
  cf.~\cite[3.11]{SingBook}.
\end{exmp}

\subsection{Positive characteristic}

For the definition of $F$-\sings appearing in this section, please refer to
\autoref{sec:fp-sings}. 

\begin{exmp}\label{exmp:F-pure}
  Let $k$ be a field of characteristic $p>0$ and let $Y$ be the curve coinsisting of
  the three coordinate axes in $\bA^3_k$, i.e., let $Y=\Spec k[x,y,z]/(xy,xz,yz)$.
  Then $Y$ is \Fp by \cite[5.38]{MR0417172} and hence it is also \Fan by
  \cite{MR3271179}.
\end{exmp}

A frequently used way to show that a class of singularities is invariant under small
deformation is to show the following two conditions:
\begin{enumerate}
\item\label{item:13} The class of \sings in question satisfies an \emph{inversion of
    adjunction type} property, i.e., if a Cartier divisor $Y\subseteq X$ belongs to
  this class, then so does $X$.
\item\label{item:14} The class of \sings in question satisfies a \emph{Bertini type}
  property, i.e., if $X$ belongs to this class and $Y\subseteq X$ is a general member
  of a very ample linear system, then $Y$ also belongs to this class.
\end{enumerate}
It is easy to see that these two conditions imply that if in a flat family a fiber
belongs to the given class of \sings, then so do nearby fibers. 

This method indeed proves that $\bQ$-Gorenstein \Fp \sings defined over an
algebraically closed field are invariant under small deformation. Property
\autoref{item:13}, the inversion of adjunction type property, for Gorenstein \Fp
\sings holds by \cite[3.4(2)]{MR701505} and property \autoref{item:14}, the Bertini
type property, at least over an algebraically closed field holds by
\cite{MR3108833}. The $\bQ$-Gorenstein case of \autoref{item:13} can be proved using
\cite[7.2]{MR2587408}.

Similarly, $F$-injective singularities with \toploccohs satisfy \autoref{item:13}
\cite{MR3263925}.

Without the $\bQ$-Gorenstein assumption \Fp \sings do not satisfy \autoref{item:13}
\cite{MR701505,MR1693967}. In contrast, \Fan \sings satisfy \autoref{item:13}
\cite{ma_quy_2017}, but it is not known at the moment whether they satisfy
\autoref{item:14}.

The fact that \Fan \sings satisfy \autoref{item:13}, but \Fp \sings in general do
not, leads to simple examples of \Fan \sings that are not \Fp:

\begin{exmp}\label{exmp:F-anti-nilp}\cite{MR701505,MR1693967,MR3649223,ma_quy_2017}
  Let $X=\Spec k[x,y,z,t]/(xy,xz,y(z-t^2)$ and $Y=(t=0)\subseteq X$. Then
  $Y\simeq \Spec k[x,y,z]/(xy,xz,yz)$ and hence it is \Fp by
  \autoref{exmp:F-pure}. Furthermore, then $X$ is also \Fan by
  \cite[4.2]{ma_quy_2017}, but it is not \Fp by \cite[3.2]{MR1693967}.
\end{exmp}

\begin{exmp}[\emph{Singularities of cones III}]\label{exmp:sing-of-cones-ii}
  We have seen in \autoref{cor:cone-over-abelian} that in characteristic $0$ a cone
  over an abelian variety has \DB \sings. We have a similar statement in positive
  characteristic: Let $X$ be an ordinary abelian variety over a field of positive
  characteristic and $\sL$ an ample line bundle on $X$. Then $C_a(X,\sL)$ has \Fp
  singularities by \cite[Lemma~1.1]{MR916481}
  and hence \Fan \sings by \cite{MR3271179}.
\end{exmp}

\section{Filtrations on modules over Artinian local rings}\label{sec:filtrations}
\noin
We will use the following notation throughout.

\begin{demo-r}{Maximal filtrations}\label{notation}
  Let $(S,\frm, k)$ be an Artinian local ring and $N$ a finite $S$-module with a
  filtration
  $N= N_0\supsetneq N_1\supsetneq \dots \supsetneq N_{q} \supsetneq N_{q+1}=0$ such
  that $\factor {N_{j}}{N_{j+1}} \simeq k$ as $S$-modules for each $j=0,\dots,q$.
  Further let $f:(X,x)\to (\Spec S,\frm)$ be a flat local morphism and denote the
  fiber of $f$ over $\frm$ by $X_{\frm}$. It follows that then for each
  $j=0,\dots,q$,
  \begin{equation}
    \label{eq:4}
    f^*\left( \factor {N_{j}}{N_{j+1}} \right) \simeq \sO_{X_{\frm}}.
  \end{equation}
\end{demo-r}
\begin{demo-r}{Filtering $S$}\label{filt-S}
  In particular, considering $S$ as a module over itself, we choose a filtration of
  $S$ by ideals
  $S=I_0\supsetneq I_1\supsetneq \dots \supsetneq I_{q} \supsetneq I_{q+1}=0$ such that
  $\factor {I_{j}}{I_{j+1}} \simeq k$ as $S$-modules for all $0\leq j\leq q$.
  Observe that in this case $I_{1}=\frm$ and for every $j$ there exists a
  $t_j\in I_j$ such that the composition
  $\xymatrix{S\ar[r]^-{t_j \cdot} & I_j \ar[r] & \factor {I_j}{I_{j+1}}}$ induces an
  isomorphism $\factor S\frm\simeq \factor {I_j}{I_{j+1}}$. In particular,
  $\ann\left(\factor {I_j}{I_{j+1}}\right) =\frm$. Finally, let
  $S_j:=\factor S{I_j}$. Note that $S_1=\factor S\frm$ and $S_{q+1}=S$.
\end{demo-r}
\begin{demo-r}{Filtering $\omega_S$}\label{filt-o}
  Applying Grothendieck duality to the closed embedding given by the surjection
  $S\onto S_j$ implies that $\omega_{S_j} \simeq \Hom_S(S_j, \omega_S)$ and we obtain
  injective $S$-module homomorphisms
  $\varsigma_j: \omega_{S_j} \into \omega_{S_{j+1}}$ induced by the natural
  surjection $S_{j+1}\onto S_j$.  Using the fact that $\omega_S$ is an injective
  $S$-module and applying the functor $\Hom_S(\blank, \omega_S)$ to the short exact
  sequence of $S$-modules
  \[
  \xymatrix{%
    0 \ar[r] & \factor{I_j}{I_{j+1}} \ar[r] & S_{j+1} \ar[r] & S_j \ar[r] & 0,
  }
  \]
  we obtain another short exact sequence of $S$-modules:
  \begin{equation}
    \label{eq:10}
    \xymatrix{%
      0 \ar[r] & \omega_{S_j}
      \ar[r]^-{\varsigma_j} & \omega_{S_{j+1}} \ar[r] & \Hom_S\left(k, 
        \omega_S\right )\simeq k \ar[r] & 0.
    }
  \end{equation}
  Therefore we obtain a filtration of $N=\omega_S$ by the submodules
  $N_j:=\omega_{S_{q+1-j}}$ as in \eqref{notation} where
  $q+1=\length_S(S)=\length_S(\omega_S)$.  The composition of the embeddings in
  \autoref{eq:10} will be denoted by
  $\varsigma:=\varsigma_{q}\circ\dots\circ\varsigma_1 :\omega_{S_1}\into
  \omega_{S_{q+1}}= \omega_{S}$.
\end{demo-r}

Recall that the \emph{socle} of a module $M$ over a local ring $(S,\frm,k)$ is
\begin{equation}
  \label{eq:16}
  \Soc M :=(0:\frm)_M = \{x\in M \mid \frm\cdot x =0 \} \simeq \Hom_S(k, M).
\end{equation}
$\Soc M$ is naturally a $k$-vector space and $\dim_k\Soc\omega_S=1$ by the
definition of the canonical module. In particular, $\Soc\omega_S\simeq k$ and this is
the only $S$-submodule of $\omega_S$ isomorphic to $k$.

\begin{lem}\label{lem:I-times-omega}
  Using the notation from \eqref{filt-S} and \eqref{filt-o}, we have that 
  \begin{equation}
    \label{eq:12}
    \im \varsigma = \Soc \omega_{S} = I_{q} \omega_{S}.
  \end{equation}
\end{lem}

\begin{subrem}
  Note that we are not simply stating that these modules in \autoref{eq:12} are
  isomorphic, but that they are equal as submodules of $\omega_{S}$.
\end{subrem}

\begin{proof}
  Since $S_{1}\simeq \factor S\frm\simeq k$, and hence $\omega_{S_{1}}\simeq k$, it
  follows that the image of the embedding $\varsigma :\omega_{S_1}\into \omega_{S}$
  maps $\omega_{S_1}$ isomorphically onto $\Soc \omega_{S}$:
  \begin{equation}
    \label{eq:18}
    \xymatrix{%
      \im \varsigma 
      = \Soc \omega_{S}.
    }
  \end{equation}
  As $\omega_{S}$ is a dualizing sheaf, $I_{q}\omega_{S}\neq 0$, and since
  $I_{q}\simeq \factor S\frm$ it follows that
  \[
  0\neq I_{q}\omega_{S} \subseteq (0:\frm)_{\omega_{S}} =\Soc\omega_{S}\simeq k.
  \]
  Since $k$ is a simple $S$-module, this implies that
  $I_{q}\omega_{S} =\Soc\omega_{S}$ which proves \autoref{eq:12}.
\end{proof}

\section{Families over Artinian local rings}\label{sec:db-families-over}

\noin
We will frequently use the following notation.

\begin{notation}\label{not:top-loc-cohs}
  Let $A$ be a noetherian ring, $(R,\frm)$ a noetherian local $A$-algebra,
  $I\subset R$ a nilpotent ideal and $(T,\frn)\leteq (R/I,\frm/I)$ with natural
  morphism $\alpha:R\onto T$.
\end{notation}

\begin{defini}\label{def:top-loc-cohs-for-schemes}
  Recall from \autoref{def:liftable-cohom} that we say that $(T,\frn)$ has
  \emph{\toploccohs over $A$} if for any $(R,\frm)$ as in \autoref{not:top-loc-cohs},
  the induced homomorphism on local cohomology $H^i_\frm(R) \onto H^i_{\frn}(T)$ is
  surjective for all $i$.
  
  We extend this definition to schemes: Let $(X,x)$ be a local scheme over a
  noetherian ring $A$. Then we say that $(X,x)$ has \emph{\toploccohs over $A$} if
  $\sO_{X,x}$ has \toploccohs over $A$. If $f:X\to B$ is a morphism of 
  schemes then we say that $X$ has \emph{\toploccohs over $B$} if $(X,x)$ has
  \toploccohs over $A$ for each $x\in X$ and for each $\Spec A\subseteq B$ open
  affine neighbourhood of $f(x)\in B$.
\end{defini}

\begin{rem}\label{rem:top-loc-cohs-is-hereditary}
  A simple consequence of the definition is that if $X$ has \toploccohs over a scheme
  $Z$, then for any morphism $g:X\to B$ of $Z$-schemes $X$ has \toploccohs over $B$
  as well. In particular, if $X$ has \toploccohs over a field $k$, then it has
  \toploccohs over any other $k$-scheme to which it admits a map. In addition if
  $\kar k=0$, then $X$ has \toploccohs by \autoref{rem:inheriting-top-loc-cohs}.
\end{rem}

\noin 
Next we need a simple lemma regarding \toploccohs:

\begin{lem}\label{lem:surj-for-R-implies-surj-for-M}\label{loc-coh-surj-2}
  Using \autoref{not:top-loc-cohs} let $M$ be an $R$-module such that there exists a
  surjective morphism $M\onto T$. Assume that the induced natural homomorphism
  $H^i_\frm(R) \onto H^i_\frn(T)$ is surjective for some $i\in\bN$.  Then the induced
  homomorphism on local cohomology
  \begin{equation}
    \label{eq:32}
    \xymatrix{%
      H^i_\frm(M) \ar@{->>}[r] & H^i_\frm(T)\simeq H^i_\frn(T) }
  \end{equation}
  is surjective for the same $i$. In particular, if $(T,\frn)$ has \toploccohs over
  $A$, then the homomorphism in \autoref{eq:32} is surjective for every $i\in\bN$.
\end{lem}

\begin{proof}
  Let $t\in M$ be such that $\alpha(t)=1\in T$ and let $\beta: R\to M$ be defined by
  $1\mapsto t$.  Then $\alpha\circ\beta=\alpha :R \to T$ is the natural quotient
  morphism, hence the surjective morphism $H^i_\frm(R) \onto H^i_\frm(T)$ factors
  through $H^i_\frm(M)$ which proves the statement.
\end{proof}

\begin{prop}\label{thm:loc-coh-inj}
  Let $(S,\frm,k)$ be an Artinian local ring and $f:(X,x)\to (\Spec S,\frm)$ a flat
  local morphism. Let $N$ be a finite $S$-module with a filtration as in
  \eqref{notation} and assume that $(X_{\frm},x)$ has \toploccohs over $S$.  Then for
  each $i, j$, the natural sequence of morphisms induced by the embeddings
  $N_{j+1}\into N_j$ forms a short exact sequence,
  \[
  \xymatrix{%
    0 \ar[r] & H^i_x(f^*N_{j+1}) \ar[r] & H^i_x(f^*N_{j}) \ar[r] & H^i_x\left(
      f^*\left(\factor {N_{j}}{N_{j+1}}\right) \right)\simeq
    H^i_x\left(\sO_{X_{\frm}} \right) \ar[r] & 0. }
  \]
\end{prop}

\begin{proof}
  Since $\ann\left(\factor {N_{j}}{N_{j+1}}\right) =\frm$, there is a natural
  surjective morphism
  \[
  f^*N_{j}\otimes \sO_{X_{\frm}}\onto f^*\left(\factor {N_{j}}{N_{j+1}}\right).
  \]
  By \autoref{loc-coh-surj-2} and \autoref{eq:4}, the natural homomorphism
  \begin{equation}
    \label{eq:5}
    \xymatrix{%
      H^i_x(f^*N_{j}) \ar@{->>}[r] & H^i_x\left(
        f^*\left(\factor {N_{j}}{N_{j+1}}\right)
      \right)\simeq 
      H^i_x\left(\sO_{X_{\frm}} \right) }
  \end{equation}
  is surjective for all $i$.
  Since $f$ is flat, we have a short exact sequence for every $j>0$:
  \[
  \xymatrix{%
    0 \ar[r] & f^*N_{j+1} \ar[r] & f^*N_{j} \ar[r] &  f^*\left(\factor
      {N_{j}}{N_{j+1}}\right) \ar[r] & 0,
  }
  \]
  and hence the statement follows from \autoref{eq:5}.
\end{proof}

\begin{demo-r}{The exceptional inverse image of the structure
    sheaves}\label{uppershriek}%
  Let $(S,\frm,k)$ be an Artinian local ring with a filtration by ideals as in
  \eqref{filt-S}. Further let $f:X\to \Spec S$ be a flat morphism that is essentially
  of finite type and $f_j=f\resto{X_j}:X_j:=X\times_{\Spec S}\Spec S_j\to \Spec S_j$
  where $S_j=S/I_j$ as defined in \eqref{filt-S}, e.g., $X_{q+1}=X$ and
  $X_{1}=X_{\frm}$, the fiber of $f$ over the closed point of $S$. By a slight abuse
  of notation we will denote $\omega_{\Spec S}$ with $\omega_S$ as well, but it will
  be clear from the context which one is meant at any given time.

  Using the description of the exceptional inverse image functor via the
  residual/dualizing complexes \cite[(3.3.6)]{Conrad00} (cf.\cite[3.4(a)]{RD},
  \cite[\href{http://stacks.math.columbia.edu/tag/0E9L}{Tag 0E9L}]{stacks-project}):
  \begin{equation}
    \label{eq:21}
    f^! = \myR\sHom_X(\myL f^*\myR\sHom_S(\blank, \dcx S), \dcx X) 
  \end{equation}
  and the facts that $S$ is Artinian and $f$ is flat, we have that
  \begin{equation*}
    \dcx{X_j/S_j}\simeq f_j^!\sO_{\Spec S_j}\simeq 
    \myR\sHom_{X_j}(f_j^*\omega_{S_j},    \dcx {X_j}).
  \end{equation*}
  By Grothendieck duality
  \[
  \myR\sHom_{X_j}(f_j^*\omega_ {S_j}, \dcx {X_j}) \simeq
  \myR\sHom_{X}(f_j^*\omega_ {S_j}, \dcx {X}),
  \]
  and as $f_j^*\omega_ {S_j}=f^*\omega_ {S_j}$ and
  $\omega_{S_j}\simeq \Hom_S(S_j, \omega_S)\simeq \myR\sHom_S(\sO_{\Spec S_j}, \dcx
  S)$ we obtain that
  \begin{equation}
    \label{eq:20}
    \dcx{X_j/S_j}\simeq \myR\sHom_X(f^*\omega_ {S_j}, \dcx X)\simeq f^!\sO_{\Spec S_j},
  \end{equation}
  in particular, that
  \begin{equation}
    \label{eq:9}
    \dcx{X_{\frm}}\simeq f^!k \simeq \myR\sHom_X(f^*\Hom_S(k,\omega_ {S}), \dcx X) 
    \simeq \myR\sHom_X(\sO_{X_\frm}, \dcx X).
  \end{equation}
\end{demo-r}

\begin{demo-r}{Natural morphisms of dualizing complexes}\label{nat-morph}%
  We will continue using the notation from \eqref{uppershriek}.  Applying $f^!$ to
  the natural surjective morphism $\xymatrix{S_{j+1} \ar@{->>}[r] & S_j}$ gives a
  natural morphism
  \begin{equation}
    \label{eq:17}
    \xymatrix{%
      \varrho_j: \dcx {X_{j+1}/S_{j+1}} \ar[r] & \dcx {X_{j}/S_{j}}. }
  \end{equation}
  Notice that $\varrho_j$ is Grothendieck dual to $f^*\varsigma_j$ defined in
  \eqref{filt-S}.  Indeed, $\varsigma_j$ is obtained by applying
  $\Hom_S(\blank,\omega_S)$ to the morphism $\xymatrix{S_{j+1} \ar@{->>}[r] & S_j}$,
  and then $\varrho_j$ is obtained by applying $f^*$ and then
  $\myR\sHom_X(\blank,\dcx X)$. Notice further that $\sfh^{-i}(\varrho_j)$ factors
  through the natural base change morphism of \autoref{prop:base-change-map} for each
  $i\in\bZ$.

  The composition of the surjective morphisms $\xymatrix{S_{j+1} \ar@{->>}[r] & S_j}$
  for all $j$ is the natural surjective morphism
  $\xymatrix{S \ar@{->>}[r] & \factor S{\frm} \simeq k}$, and hence the composition
  of the $\varrho_j$'s gives the natural morphism
  \begin{equation}
    \label{eq:19}
    \xymatrix{%
      \varrho:= \varrho_1\circ\dots \circ\varrho_q
      : \dcx {X/S} \ar[r] & \dcx {X_{1}/S_{1}} =\dcx{X_\frm},}
  \end{equation}
  which is then Grothendieck dual to
  $f^*\varsigma:=f^*(\varsigma_q\circ\dots\circ\varsigma_1)$ and
  $\sfh^{-i}(\varrho_j)$ factors through the natural base change morphism of
  \autoref{prop:base-change-map} for each $i\in\bZ$.
\end{demo-r}

\noin
In the rest of this section we will use the following notation and assumptions. 

\begin{assume}\label{ass:DB}
  Let $(S,\frm,k)$ be an Artinian local ring and $f:(X,x)\to (\Spec S,\frm)$ a flat
  local morphism that is essentially of finite type. Assume that $(X_{\frm}, x)$,
  where $X_\frm$ is the fiber of $f$ over the closed point of $\Spec S$, has
  \toploccohs over $S$. Note that by definition $x\in X_{\frm}$ and that we will keep
  using the notation introduced in \autoref{eq:17} and \autoref{eq:19}.
\end{assume}

\begin{thm}\label{thm:surjectivity}
  For each $i,j\in\bN$,
  \begin{enumerate}
  \item\label{item:1} the natural morphism $\xymatrix{%
      \sfh^{-i}(\varrho_j) : \sfh^{-i}(\dcx {X_{j+1}/S_{j+1}}) \ar@{->>}[r] &
      \sfh^{-i}(\dcx {X_{j}/S_{j}}) }$ is surjective,
  \item\label{item:2} the natural morphism $\xymatrix{%
      \sfh^{-i}(\varrho) : \sfh^{-i}(\dcx {X/S}) \ar@{->>}[r] & \sfh^{-i}(\dcx
      {X_{\frm}}) }$ is surjective,
  \item\label{item:3} the natural morphisms induced by $\varrho_j$ form a \ses,
    \[
    \xymatrix@C4em{%
      0 \ar[r] & \sfh^{-i}(\dcx {X_\frm}) \ar[r] & \sfh^{-i}(\dcx {X_{j+1}/S_{j+1}})
      \ar[r]^-{\sfh^{-i}(\varrho_j)} & \sfh^{-i}(\dcx {X_j/S_j}) \ar[r] & 0,}
    \]
  \item\label{item:5}
    $\ker \sfh^{-i}(\varrho_j) = I_{j}\sfh^{-i}(\dcx {X_{j+1}/S_{j+1}}) \simeq
    \factor {I_{j}\sfh^{-i}(\dcx {X/S})}{I_{j+1}\sfh^{-i}(\dcx {X/S})}$,
  \item\label{item:4}
    $\sfh^{-i}(\dcx {X_{j}/S_{j}})\simeq \factor {\sfh^{-i}(\dcx {X/S})}
    {I_{j}\sfh^{-i}(\dcx {X/S})}\simeq {\sfh^{-i}(\dcx {X/S})} \otimes_{\sO_X}
    \sO_{X_j}$, and
  \item\label{item:11} $\ker \sfh^{-i}(\varrho) = \frm\sfh^{-i}(\dcx {X/S})$.
  \end{enumerate}
\end{thm}

\begin{proof}
  Let $N=\omega_S$ and consider the filtration on $N$ given by
  $\omega_{S_j}=N_{{q+1-j}}$, cf.~\eqref{filt-o}, \autoref{eq:10}.
  %
  Further let 
  $(\ )\what{\ }$ denote the completion at $x$ (the closed point of $X$).  Then by
  \autoref{thm:loc-coh-inj}, for all $i, j\in\bN$, there exists a short exact
  sequence
  \begin{equation}
    \label{eq:8}
    \xymatrix{%
      0 \ar[r] & H^i_x(f^*\omega_{S_{j}}) \ar[r] & H^i_x(f^*\omega_{S_{j+1}}) \ar[r] &
      H^i_x\left(f^*\left(\factor {\omega_{S_{j+1}}}{\omega_{S_{j}}}\right)\right)
      \ar[r] & 0. }    
  \end{equation}
  Applying local duality \cite[Corollary~V.6.5]{RD} to \autoref{eq:8} gives the short
  exact sequence
  \[
  \xymatrix@C1.5em{%
    0 \ar[r] & \sExt^{-i}_{X} \left(f^*\left(\factor
        {\omega_{S_{j+1}}}{\omega_{S_{j}}}\right), 
      \dcx{X} \right) \what{\vphantom{\big)}} \ar[r] &
    \sExt^{-i}_{X}(f^*\omega_{S_{j+1}}, \dcx{X} )\ \what{} \ar[r] &
    \sExt^{-i}_{X}(f^*\omega_{S_{j}}, \dcx{X} )\ \what{} \ar[r] & 0. }
  \]
  for all $i, j\in\bN$.  Since completion is faithfully flat
  \cite[\href{http://stacks.math.columbia.edu/tag/00MC}{Tag 00MC}]{stacks-project},
  this implies that there are short exact sequences
  \begin{align}
    \label{eq:11}
    \begin{split}
      \xymatrix{%
        0 \ar[r] & \sExt^{-i}_{X} \left(f^*\left(\factor
            {\omega_{S_{j+1}}}{\omega_{S_{j}}}\right), 
          \dcx{X} \right) \ar[r] & \hskip6em& \hskip6em } \\ \xymatrix{%
        && & \ar[r] & \sExt^{-i}_{X}\left(f^*\omega_{S_{j+1}}, \dcx{X} \right) 
        \ar[r] & \sExt^{-i}_{X}\left(f^*\omega_{S_{j}}, \dcx{X} \right)  \ar[r] &
        0. }
    \end{split}
  \end{align}
  Recall that
  $\sExt^{-i}_{X}\left(f^*\omega_{S_{j}}, \dcx{X} \right)\simeq \sfh^{-i}(\dcx
  {X_{j}/S_{j}})$ for each $i,j$,
  by \autoref{eq:20}. 
  Further observe that the surjective morphism in \autoref{eq:11} is the
  $-i^\text{th}$ cohomology sheaf of the Grothendieck dual of $f^*\varsigma_{j}$ and
  hence via the above isomorphisms, it corresponds to
  $\sfh^{-i}(\varrho_{j})$. Therefore \autoref{eq:11} implies \autoref{item:1}.
  By \autoref{eq:10}
  $f^*\left(\factor {\omega_{S_{j+1}}}{\omega_{S_{j}}}\right)\simeq \sO_{X_\frm}$,
  and hence
  $\sExt^{-i}_{X} \left(f^*\left(\factor {\omega_{S_{j+1}}}{\omega_{S_{j}}}\right),
    \dcx{X} \right)\simeq \sfh^{-i}(\dcx {X_\frm})$,
  so \autoref{eq:11} also implies \autoref{item:3}.  Composing the surjective
  morphisms in \autoref{eq:11} for all $j$ implies that the natural morphism
  \[
  \xymatrix@C4em{%
    \sfh^{-i}(\dcx {X/S})\simeq \sExt^{-i}_{X}\left(f^*\omega_S, \dcx{X} \right)
    \ar@{->>}[r]^-{\sfh^{-i}(\varrho)} & \sExt^{-i}_{X}\left(f^*\omega_{S_q}, \dcx{X}
    \right) \simeq \sfh^{-i}(\dcx {X_{\frm}}) }
  \]
  is surjective and hence \autoref{item:2} follows as well.

  Similarly, composing the injective maps in \autoref{eq:8} for all $j$ shows that
  the embedding $\varsigma: \omega_{S_1}\into \omega_S$ induces an embedding on local
  cohomology:
  \begin{equation}
    \label{eq:15}
    H^i_x(f^*\omega_{S_1}) \subseteq
    H^i_x(f^*\omega_{S}).
  \end{equation}
  Next we prove \autoref{item:5} for $j=q$ first.  Since
  $\sfh^{-i}(\dcx {X_{q}/S_{q}})$ is supported on $X_{q}$ it follows that
  \[
  I_{q}\sfh^{-i}(\dcx {X/S})\subseteq K:= \ker \sfh^{-i}(\varrho_{q})
  \]
  %
  Recall from \eqref{filt-S} that there exists a $t_{q}\in I_{q}$ such that
  $I_{q}=St_{q}\simeq \factor S\frm$ and from \autoref{lem:I-times-omega} that
  $I_{q}\omega_S=\Soc \omega_S$. It follows that for $j=q$ the short exact sequence
  of \autoref{eq:10} takes the form
  \begin{equation}
    \label{eq:6}
    \xymatrix@C3em{%
      0 \ar[r] & \omega_{S_{q}} \ar[r] & \omega_{S} \ar[r]^-{\tau} &
      \Soc\omega_S \ar[r] & 0,  }
  \end{equation}
  where $\tau:\omega_S\onto\Soc\omega_S\subset \omega_S$ may be identified with
  multiplication by $t_{q}$ on $\omega_S$.  Applying $f^*$ and taking local
  cohomology we obtain the short exact sequence
  \begin{equation}
    \label{eq:14}
    \xymatrix@C3em{%
      0 \ar[r] & H^i_x(f^*\omega_{S_{q}}) \ar[r] & H^i_x(f^*\omega_{S})
      \ar[r]^-{H^i_x(\tau)} & H^i_x\left(f^*\Soc\omega_S \right) \ar[r] & 0, }
  \end{equation}
  which is of course just \autoref{eq:8} for $j=q$.  Clearly, 
  the morphism $H^i_x(\tau)$ may also be identified with multiplication by $t_{q}$ on
  $H^i_x(f^*\omega_S)$.  By \autoref{lem:I-times-omega} and \autoref{eq:15}, the
  natural morphism
  $H^i_x(\varsigma): H^i_x\left(f^*\Soc\omega_S \right) =H^i_x(I_qf^*\omega_S) =
  H^i_x(f^*\omega_{S_1}) \to H^i_x(f^*\omega_S)$
  is injective.  Since $H^i_x(\tau)$, i.e., multiplication by $t_q$ on
  $H^i_x(f^*\omega_{S})$, is surjective onto $H^i_x\left(f^*\Soc\omega_S \right)$, it
  follows that
  \begin{equation}
    \label{eq:25}
    \xymatrix@R5em{%
      H^i_x\left(f^*\Soc\omega_S \right) \ar[r]_-{H^i_x(\varsigma)}^-{\simeq} & 
      \im H^i_x(\varsigma)       = I_{q} H^i_x(f^*\omega_{S}) \ar@{^(->}[r] & 
      H^i_x(f^*\omega_{S}),
    }
  \end{equation}
  i.e., $H^i_x\left(f^*\Soc\omega_S \right)$ coincides with
  $I_{q} H^i_x(f^*\omega_{S})$ as submodules of $H^i_x(f^*\omega_{S})$.
  Next let $E$ be an injective hull of $\kappa(x)=\factor{\sO_{X,x}}{\frm_{X,x}}$ and
  consider a morphism $\phi: H^i_x(f^*\Soc\omega_S)\to E$. As $E$ is injective,
  $\phi$ extends to a morphism $\wt\phi: H^i_x(f^*\omega_S)\to E$.  If
  $a\in H^i_x(f^*\omega_S)$, then
  $t_{q}a\in I_{q}H^i_x(f^*\omega_S)=H^i_x\left(f^*\Soc\omega_S \right)$, so
  \[
  t_{q}\wt\phi(a)=\wt\phi(t_{q}a)=\phi(t_{q}a)=\left(\phi\circ
    H^i_x(\tau)\right)(a)
  \]
  Therefore, $\phi\circ H^i_x(\tau)= t_{q}\wt\phi$. 
  Similarly, if $\psi: H^i_x(f^*\omega_S)\to E$ is an arbitrary morphism, then
  setting $\phi=\psi\resto{H^i_x(f^*\Soc\omega_S)}: H^i_x(f^*\Soc\omega_S)\to E$ and
  applying the same computation as above, with $\wt\phi$ replaced by $\psi$, shows
  that $\phi\circ H^i_x(\tau)= t_{q}\psi$.
  It follows that the embedding induced by $H^i_x(\tau)$,
  \begin{equation}
    \label{eq:28}
    \alpha: \Hom_{\sO_{X,x}}(H^i_x(f^*\Soc\omega_S), E)\into
    \Hom_{\sO_{X,x}}(H^i_x(f^*\omega_S), E)
  \end{equation}
  identifies $\Hom_X(H^i_x(f^*\Soc\omega_S), E)$ with
  $I_{q}\Hom_X(H^i_x(f^*\omega_S), E)$. By local duality this implies that 
  \[
  \left(\factor{\ker \left[ \sfh^{-i}(\varrho_{q}): \sfh^{-i}(\dcx {X/S})\onto
        \sfh^{-i}(\dcx {X_{q}/S_{q}}) \right]}{I_{q}\sfh^{-i}(\dcx {X/S})}\right)
  \otimes \what{\sO}_{X,x}=0
  \]
  and hence, since completion is faithfully flat, this implies \autoref{item:5} in
  the case $j=q$.  
  Running through the same argument with $S$ replaced by $S_{j+1}$ gives the equality
  in \autoref{item:5} for all $j$.  In addition, \autoref{item:5} for $j=q$ also
  implies \autoref{item:4} for $j\geq q$.
  Assuming that \autoref{item:4} holds for $j=r+1$ implies the isomorphism in
  \autoref{item:5} for $j=r$. In turn, the entire \autoref{item:5} for $j=r$,
  combined with \autoref{item:4} for $j=r+1$, implies \autoref{item:4} for
  $j=r$. Therefore, \autoref{item:5} and \autoref{item:4} follow by descending
  induction on $j$ and then \autoref{item:11} follows from \autoref{item:5} and the
  definition of $\varrho$.
\end{proof}

\noin
Next we need a simple lemma.

\begin{lem}\label{lem:simple}
  Let $R$ be a ring. $M$ an $R$-module, $t\in R$ and $J=(t)\subseteq R$. Assume that
  $(0:J)_M=(0:J)_R\cdot M$. Then the natural morphism
  $\xymatrix{J\otimes_R M\ar[r]^-\simeq & JM}$ is an isomorphism.
\end{lem}

\begin{proof}
  This natural morphism is always surjective. Suppose $m\in M$ is such that
  $t\otimes m\mapsto 0$ via this morphism. In other words such that $tm=0$. This
  means, by definition, that $m\in (0:J)_M$ and hence by assumption there exist
  $y\in (0:J)_R\subseteq R$ and $m'\in M$ such that $m=ym'$. Then
  $t\otimes m=t\otimes ym'=yt\otimes m'=0$, since $yt=0$. This proves the claim.
\end{proof}

\begin{prop}\label{prop:tensor}
  Using the same notation as above,
  \begin{enumerate}
  \item\label{item:10}
    $I_j\otimes \sfh^{-i}(\dcx {X/S})\simeq I_{j}\sfh^{-i}(\dcx {X/S})$,
  \item\label{item:6} for any $l\in\bN$,
    $\factor{I_j}{I_{j+l}}\otimes \sfh^{-i}(\dcx {X/S})\simeq \factor
    {I_{j}\sfh^{-i}(\dcx {X/S})}{I_{j+l}\sfh^{-i}(\dcx {X/S})}$, and
  \item\label{item:12} for any $l\in\bN$,
    $\factor{\frm^l}{\frm^{l+1}}\otimes \sfh^{-i}(\dcx {X/S})\simeq \factor
    {\frm^l\sfh^{-i}(\dcx {X/S})}{\frm^{l+1}\sfh^{-i}(\dcx {X/S})}$.
  \end{enumerate}
\end{prop}

\begin{proof}
  Notice that since $H^i_x(f^*\Soc\omega_S)$ is both a quotient and a submodule of
  $H^i_x(f^*\omega_S)$, there are two natural maps between
  $\Hom_{\sO_{X,x}}(H^i_x(f^*\Soc\omega_S), E)$ and
  $\Hom_{\sO_{X,x}}( H^i_x(f^*\omega_S), E)$.  Regarding $H^i_x(f^*\Soc\omega_S)$ a
  quotient module via $H^i_x(\tau)$ we get the embedding
  $\alpha=(\blank)\circ H^i_x(\tau)$ in \autoref{eq:28}, and considering it a
  submodule the restriction map
  \begin{equation*}
      \xymatrix@R0em{%
        \beta: \Hom_{\sO_{X,x}}( H^i_x(f^*\omega_S), E) \ar[r] &
        \Hom_{\sO_{X,x}}(H^i_x(f^*\Soc\omega_S), E). \\
        \phi \ar@{|->}[r] & \phi\resto{H^i_x(f^*\Soc\omega_S)} }
  \end{equation*}
  These maps are of course not inverses to each other. In fact, we have already
  established (cf.~\autoref{eq:28}) that
  $\phi\resto{H^i_x(f^*\Soc\omega_S)}\circ H^i_x(\tau)= t_{q}\phi$ and hence the
  composition $\alpha\circ\beta$ is just multiplication by $t_q$:
  \begin{equation}
    \label{eq:24}
    \begin{aligned}
      \xymatrix{%
        & %
        \ar@{}[l]^(.675){}="a" %
        \ar@{} "a";[d]^(.05){}="b"%
        \ar[rd]^(.55){\alpha\circ\beta}
        \phi\in \Hom_{\sO_{X,x}}(H^i_x(f^*\omega_S), E) \ar[r]^-\beta &
        \Hom_{\sO_{X,x}}(H^i_x(f^*\Soc\omega_S), E) \ar[d]^-\alpha_-\simeq \\
        && %
        \ar@{}[l]^(.45){}="c" %
        \ar@{} "c";[u]^(.15){}="d"%
        \ar@/_1em/@{|->} "b";"c" %
        t_q \phi\in I_q\Hom_{\sO_{X,x}}(H^i_x(f^*\omega_S), E). }
    \end{aligned}
  \end{equation}
  This implies, (cf.~\autoref{eq:15} and \autoref{eq:25}), that $\sfh^{-i}(\varrho)$
  may be identified with multiplication by $t_q$ on $\sfh^{-i}(\dcx {X/S})$.
  Together with \autoref{thm:surjectivity}\autoref{item:11} this implies that
  \[
  (0:I_q)_{\sfh^{-i}(\dcx {X/S})} =\ker\sfh^{-i}(\varrho)= \frm \sfh^{-i}(\dcx {X/S})
  = (0:I_q)_S\cdot \sfh^{-i}(\dcx {X/S}),
  \]
  and hence the natural morphism
  \begin{equation}
    \label{eq:26}
    \xymatrix{%
      I_q\otimes_S \sfh^{-i}(\dcx {X/S}) \ar[r]^-\simeq & I_q\sfh^{-i}(\dcx {X/S}) 
    }
  \end{equation}
  is an isomorphism by \autoref{lem:simple}.
  Now assume, by induction, that \autoref{item:10} holds for $S_q$ in place of
  $S$. In particular, keeping in mind that $S_q=\factor S{I_q}$, the natural map
  \begin{equation}
    \label{eq:30}
    \xymatrix{%
      \factor{I_{j}}{I_{q}}\otimes_{S_q} 
      \sfh^{-i}(\dcx {X_q/S_q})\ar[r]^-\simeq & 
      \left(\factor{I_{j}}{I_{q}}\right)\sfh^{-i}(\dcx {X_q/S_q}) 
    }
  \end{equation}
  is an isomorphism for all $j$.
  Consider the \ses (cf.~\autoref{thm:surjectivity}\autoref{item:4}),
  \begin{equation*}
    \xymatrix{%
      0 \ar[r] & I_q \sfh^{-i}(\dcx {X/S}) 
      \ar[r] & \sfh^{-i}(\dcx {X/S}) \ar[r] & 
      \sfh^{-i}(\dcx {X_q/S_q})  \ar[r] &  0
    }
  \end{equation*}
  and apply $\factor{I_{j}}{I_{q}}\otimes_{S} (\blank)$.  The image of
  $\factor{I_{j}}{I_{q}}\otimes_{S} I_q \sfh^{-i}(\dcx {X/S})$ in
  $\factor{I_{j}}{I_{q}}\otimes_{S} \sfh^{-i}(\dcx {X/S})$ is $0$ and hence by
  \autoref{eq:30} the natural map
  \begin{multline*}
    \xymatrix{%
      \boxed{\factor{I_{j}}{I_{q}}\otimes_{S} \sfh^{-i}(\dcx {X/S})} \simeq
      \factor{I_{j}}{I_{q}}\otimes_{S_q} \sfh^{-i}(\dcx {X_q/S_q}) \ar[r]^-\simeq &
      \left(\factor{I_{j}}{I_{q}}\right) \sfh^{-i}(\dcx{X_q/S_q})
      \simeq }\\
    \xymatrix{%
      \ar@{}[r] & \simeq \left(\factor{I_{j}}{I_{q}}\right)
      \factor{\sfh^{-i}(\dcx{X/S})}{I_q \sfh^{-i}(\dcx {X/S})} \simeq
      \boxed{\factor{I_j\sfh^{-i}(\dcx{X/S})}{I_q \sfh^{-i}(\dcx {X/S})}}.  }
  \end{multline*}
  is an isomorphism.  This, combined with \autoref{eq:26} and the 5-lemma, implies
  \autoref{item:10}. Then \autoref{item:6} is a direct consequence of
  \autoref{item:10} and the fact that tensor product is right exact.
  
  Finally, recall, that the choice of filtration in \eqref{filt-S} was fairly
  unrestricted. In particular, we may assume that the filtration $I_\kdot$ of $S$ is
  chosen so that for all $l\in \bN$, there exists a $j(l)$ such that
  $I_{j(l)}=\frm^l$. Applying \autoref{item:6} for this filtration implies
  \autoref{item:12}.
\end{proof}

\noin The following theorem is an easy combination of the results of this section.

\begin{thm}\label{thm:key}
  Let $(S,\frm,k)$ be an Artinian local ring and $f:(X,x)\to \Spec S$ a flat local
  morphism that is essentially of finite type.  If $(X_{\frm},x)$ has \toploccohs
  over $S$, then $\sfh^{-i}(\dcx {X/S})$ is flat over $\Spec S$ for each $i$.
\end{thm}

\begin{proof}
  This follows directly from \autoref{prop:tensor}\autoref{item:12} and
  \cite[\href{http://stacks.math.columbia.edu/tag/0AS8}{Tag 0AS8}]{stacks-project}.
\end{proof}

\section{Flatness and base change}\label{sec:flatness-base-change}
\noindent
In this section we prove a rather general flatness and base change theorem for the
cohomology sheaves of the relative dualizing complex. The main essential assumption
is that the relative dualizing complex exists.

\begin{defnot}\label{def-and-not.2}\label{def:embeddable}
  For morphisms $f:X\to B$ and $\vartheta: Z\to B$, the symbol $X_Z$ will denote
  $X\times_B Z$ and $f_Z:X_Z\to Z$ the induced morphism.  In particular, for $b\in B$
  we write $X_b = f^{-1}(b)$.
  %
  %

  Let $f:X\to B$ be a morphism of locally noetherian schemes. Then $f$ is
  \emph{embeddable into a smooth morphism of dimension $N$} if there exists a smooth
  morphism
  $\pi: P\to B$ of pure relative dimension $N$ over $B$ and a closed embedding
  $\jmath: X\into P$ such that $f=\pi\circ\jmath$.
  Furthermore, $f$ is \emph{locally embeddable into a smooth morphism} if each
  $x\in X$ has a neighbourhood $x\in U_x\subseteq X$ such that $f\resto{U_x}$ is
  {embeddable into a smooth morphism of dimension $N$} for some $N\in\bN$.
  %

  Note that if $f:X\to B$ is a flat morphism that is essentially of finite type then
  it is locally embeddable into a smooth morphism and that if $f$ is flat and locally
  embeddable into a smooth morphism then it admits a relative dualizing complex by
  \cite[\href{http://stacks.math.columbia.edu/tag/0E2X}{Tag 0E2X}]{stacks-project}.
\end{defnot}

\begin{lem}\label{lem:D-duality}
  Let $(B,b)$ be a local scheme and $f:X\to B$ a flat morphism embeddable into a
  smooth morphism $P\to B$ of relative dimension $N$. Then 
  \[
  \sfh^{-i}(\dcx{X/B})\simeq \sExt^{N-i}_{P}(\sO_X,\omega_{P/B}) \quad\text{ and } \quad
  \sfh^{-i}(\dcx{X_b})\simeq \sExt^{N-i}_{P_b}(\sO_{X_b},\omega_{P_b}).
  \]
\end{lem}

\begin{proof}
  Since $P/B$ is an $N$-dimensional smooth morphism, $\dcx{P/B}= \omega_{P/B}[N]$ is
  a relative dualizing complex.
  %
  By Grothendieck duality \cite[VII.3.4]{RD},
  \begin{equation*}
    \sfh^{-i}(\dcx{X/B})\simeq 
    \sfh^{-i}(\myR\sHom_{P}(\sO_X,\omega_{P/B}[N])) \simeq
    \sExt^{N-i}_{P}(\sO_X,\omega_{P/B}).
  \end{equation*}
  The same argument implies the equivalent statement for $\sfh^{-i}(\dcx{X_b})$.
\end{proof}

\noindent
The following statement is standard. We include it for ease of reference.

\begin{prop}\label{prop:base-change-map}
  Let $f:X\to B$ be a flat morphism of schemes that admits a relative dualizing
  complex and let $Z\to B$ be a morphism. Then for each $i\in\bZ$ there exists a
  natural base change morphism,
  \[
  \varrho^{-i}_Z: \sfh^{-i}(\dcx{X/B})\otimes_B\sO_Z \longrightarrow
  \sfh^{-i}(\dcx{X_Z/Z}).
  \]
\end{prop}

\begin{proof}
  For any complex $\cmx\sfA$, tensoring with an object induces a natural morphism,
  \[
  \sfh^i(\cmx\sfA)\otimes \sfM \longrightarrow \sfh^i(\cmx\sfA\lotimes \sfM).
  \]
  Applying this to the dualizing complex gives a natural map
  \[
  \varrho^{-i}_Z: \sfh^{-i}(\dcx{X/B})\otimes_B\sO_Z \longrightarrow
  \sfh^{-i}\big(\dcx{X/B}\lotimes_B\sO_Z\big).
  \]
  But $\dcx{X/B}\lotimes_B\sO_Z\simeq \dcx{X_Z/Z}$ by the base change property of
  dualizing complexes \cite[\href{http://stacks.math.columbia.edu/tag/0E2Y}{Tag
    0E2Y}]{stacks-project}, so the statement follows.
\end{proof}

\begin{terminology}\label{terminology:base-change}
  Let $f:X\to B$ be a flat morphism of schemes and $\vartheta:Z\to B$ a
  morphism. Then for an $i\in\bZ$, we will say that $\sfh^{-i}(\dcx{X/B})$
  \emph{commutes with base change to $Z$} if the natural base change morphism
  $\varrho_Z^{-i}$ of \autoref{prop:base-change-map} is an isomorphism.
\end{terminology}

\begin{rem}\label{rem:commuting-is-local}
  Since the base change morphism is defined naturally, it can be checked locally
  whether it is an isomorphism. In other words, if $\sfh^{-i}(\dcx{X/B})$ commutes
  with base change to $Z$ locally on $X$, then it commutes with base change to $Z$.
\end{rem}

\begin{rem}    
  A simple case when the condition in \eqref{terminology:base-change} holds is if
  $f:X\to B$ has \CM fibers. In that case the only non-zero cohomology sheaf of the
  relative dualizing complex is $\sfh^{-m}(\dcx{X/B})\simeq \omega_{X/B}$ where
  $m=\dim X-\dim B$ by \cite[3.5.1]{Conrad00} and it commutes with base change by
  \cite[3.6.1]{Conrad00}.  In moduli theory typically one has to deal with
  non-Cohen-Macaulay fibers. The next example shows that for these not even
  $\omega_{X/B}$ commutes with base change.  However, we see in
  \autoref{lem:completion} that the $\sfh^{-i}(\dcx{X/B})$ commute with inverse
  limits.
\end{rem}

\begin{exmp}\label{exmp:not-commuting-with-base-change}
  Let $Y$ be a normal quasi-projective threefold with isolated singularities and a
  trivial canonical divisor. Assume that $Y$ 
  is $S_2$, but not $S_3$. For instance, a non-\CM normal threefold such as a cone
  over an abelian surface in characteristic $0$ satisfies these conditions
  cf.~\autoref{cor:cone-over-abelian}. Consider a general projection of $Y$ to a line
  and resolve the indeterminacies of the projection map. Let $X$ denote the blow-up
  of $Y$ on which this rational map becomes a morphism and let $\pi:X\to \bA^1$
  denote the resulting morphism. Note that since the projection was general we may
  assume that the birational morphism $X\to Y$ is locally isomorphic near their
  singular points. In particular, we may assume that $X$ is a normal affine threefold
  with isolated singularities and a trivial canonical divisor, which is $S_2$, but
  not $S_3$.  Observe that then
  $\sfh^{-2}(\dcx{X/\bA^1})\simeq \omega_{X/\bA^1}\simeq \sO_X$ by construction. Next
  let $z\in\bA^1$ be the image of a non-$S_3$ point of $X$. Then $X_z$ and hence
  $\sO_{X_z}$ is not $S_2$, since otherwise $X$ would be $S_3$ along $X_z$. At the
  same time $\sfh^{-2}(\dcx{X_z/\{z\}})\simeq\omega_{X_z}$ is an $S_2$ sheaf
  (cf.\cite[5.69]{KM98}, \cite[3.7.5]{Kovacs17b}) and hence cannot be isomorphic to
  $\sO_{X_z}$. This implies that $\sfh^{-2}(\dcx{X/\bA^1})$ does not commute with
  base change for the morphism $\pi:X\to \bA^1$.
\end{exmp}

\begin{notation}\label{not:mod-out-by-m-to-the-q}
  Let $f:(X,x)\to (B,b)=(\Spec S,\frm)$ be a local morphism.  Let $q\in\bN$,
  $S_q:=S/\frm^q$, $\frm_q=\frm/\frm^q$ its (unique) maximal ideal, $B_q=\Spec S_q$,
  $X_q:=X\times_BB_q$, and $f_q:(X_q,x)\to (B_q,b)$ the induced local morphism.
  Further let $\what B:=\Spec (\invlim S_q)$, the completion of $B$ at $b$ and
  $\what X:=X\times _B\what B$.
\end{notation}



\begin{lem}\label{lem:completion}
  Let $f:(X,x)\to (B,b)$ be a flat local morphism that admits a relative dualizing
  complex. Fix an $i\in\bZ$ and assume that the inverse system
  $\left(\sfh^{-i-1} (\dcx{X_q/B_q})\right)$ satisfies the Mittag-Leffler condition
  \cite[\href{http://stacks.math.columbia.edu/tag/0595}{Tag 0595}]{stacks-project}.
  Then the natural base change morphism (cf.~\autoref{prop:base-change-map}) induces
  an isomorphism:
  \[
  \invlim\left(\sfh^{-i}(\dcx{X/B})\otimes_X\sO_{X_q}\right)
  \overset\simeq\longrightarrow
  \invlim\sfh^{-i}(\dcx{X_q/B_q}), 
  \]
\end{lem}

\begin{rem}
  If the local scheme $(X_b,x)$ has \toploccohs over $B$, then the inverse system
  $\left(\sfh^{-i-1} (\dcx{X_q/B_q})\right)$ satisfies the Mittag-Leffler condition
  by \autoref{thm:surjectivity}\autoref{item:1}.
\end{rem}

\begin{proof}
  By the base change property of dualizing complexes
  \cite[\href{http://stacks.math.columbia.edu/tag/0E2Y}{Tag 0E2Y}]{stacks-project}
  there exist natural restricting morphisms,
  \[
  \dcx{X_{q+1}/B_{q+1}}\longrightarrow \dcx{X_{q+1}/B_{q+1}}
  \lotimes_{X_{q+1}}\sO_{X_q} \simeq \dcx{X_q/B_q},
  \]
  so $(\dcx{X_q/B_q})$ forms an inverse system in $D^b(X)$ and hence
  $\myR\lim\dcx{X_q/B_q}$, the derived limit of the inverse system $(\dcx{X_q/B_q})$,
  exists \cite[\href{http://stacks.math.columbia.edu/tag/0CQD}{Tag
    0CQD}]{stacks-project}.  
  Since 
  the inverse system $\left(\sfh^{-i-1} (\dcx{X_q/B_q})\right)$ satisfies the
  Mittag-Leffler condition, $\myR^1\lim \sfh^{-i-1} (\dcx{X_q/B_q})=0$ by
  \cite[\href{http://stacks.math.columbia.edu/tag/091D}{Tag
    091D(3)}]{stacks-project}.  Combined with
  \cite[\href{http://stacks.math.columbia.edu/tag/0CQE}{Tag 0CQE}]{stacks-project}
  this implies that the natural base change morphism of
  \autoref{prop:base-change-map} induces an isomorphism
  \[
  \sfh^{-i}(\myR\lim \dcx{X_q/B_q}) \overset\simeq\longrightarrow \invlim\sfh^{-i}
  (\dcx{X_q/B_q}).
  \]
  The base change property of dualizing complexes
  also applies to $\dcx{X/B}$ and hence the natural restricting morphisms induce
  isomorphisms,
  \[
  \dcx{X/B} 
  \lotimes_{X}\sO_{X_q} \simeq \dcx{X_q/B_q}.
  \]
  Then the derived completion of $\dcx{X/B}$ with respect to the ideal
  $\sJ:=f^*\frm_{B,b}=\sI_{X_b\subseteq X}\subseteq \sO_{X,x}$
  \cite[\href{http://stacks.math.columbia.edu/tag/0BKH}{Tag 0BKH}]{stacks-project} is
  isomorphic to $\myR\lim\dcx{X_q/B_q}$ constructed above.  Then the statement
  follows by \cite[\href{http://stacks.math.columbia.edu/tag/0A06}{Tag 0A06}%
  ]{stacks-project}. 
\end{proof}

\begin{rem}
  In the proof above it is important to consider the derived limit
  $\myR\lim\dcx{X_q/B_q}$ as a derived completion over $X$ and not over $B$, because
  for the cited results the $\sfh^{-i}(\dcx{X/B})$ need to be finite modules.  They
  are finite over $\sO_{X,x}$ but not necessarily over $\sO_{B,b}$.
\end{rem}

\noindent
Next we prove our main flatness and base change statement.

\begin{thm}\label{thm:generalized-flat-and-base-change}
  Let $X\to B$ be a flat morphism locally embeddable into a smooth morphism.
  Fix an $i\in\bZ$ and assume that for any Artinian scheme $Z$ and morphism $Z\to B$,
  the sheaf $\sfh^{-i}(\dcx{X_Z/Z})$ is flat over $Z$ and commutes with any 
  base change to a closed subscheme of $Z$. Then $\sfh^{-i}(\dcx{X/B})$ is flat over
  $B$ and commutes with arbitrary base change.
\end{thm}

\begin{proof}
  Since the base change morphism $\varrho_Z^{-i}$ is natural, the statement is local
  on $B$, so we may replace $B$ with a local scheme $(B,b)$. Furthermore, since both
  flatness and whether or not $\varrho_Z^{-i}$ is an isomorphism can be tested
  locally on $X$, we may also replace $X$ with a local scheme $(X,x)$, use the
  notation established in \autoref{not:mod-out-by-m-to-the-q}, assume that
  $f:(X,x)\to (B,b)$ is embeddable into a smooth morphism and apply
  \autoref{lem:D-duality}.

  Let $M_q:=\sfh^{-i}\big(\dcx {X_q/B_q}\big)$. Then by assumption $M_q$ is flat over
  $B_q$ for every $q\in\bN$ and the natural base change morphism is an isomorphism:
  \[
  M_{q+1}\otimes_{B_{q+1}}\sO_{B_q} \overset\simeq\longrightarrow M_q.
  \]
  In particular, the induced natural morphism $M_{q+1}\onto M_q$ is surjective and
  hence $(M_q)$ satisfies the Mittag-Leffler condition and $\invlim M_q$ is flat over
  $B$ by the first statement of
  \cite[\href{http://stacks.math.columbia.edu/tag/0912}{Tag 0912}]{stacks-project}.
  Furthermore, let $Q=\sO_{X_j}$ for a fixed $j\in\bN$. Then
  $M_q\otimes_XQ\simeq M_j$ for any $q\geq j$ by assumption and hence
  $\invlim(M_q\otimes_XQ)\simeq M_j$.  Then by the second statement of
  \cite[\href{http://stacks.math.columbia.edu/tag/0912}{Tag 0912}]{stacks-project},
  \begin{equation}
    \label{eq:37}
    (\invlim M_q) \otimes_X\sO_{X_j} =(\invlim M_q) \otimes_XQ \simeq
    \invlim(M_q\otimes_XQ)\simeq M_j.
  \end{equation}
  On the other hand,
  $\invlim M_q \simeq\invlim\big(\sfh^{-i}(\dcx{X/B})\otimes_X\sO_{X_q}\big)$ by
  \autoref{lem:completion} and so by
  \cite[\href{http://stacks.math.columbia.edu/tag/031C}{Tag 031C}]{stacks-project},
  \begin{equation}
    \label{eq:38}
    (\invlim M_q)\otimes_X\sO_{X_j}\simeq \sfh^{-i}(\dcx{X/B})\otimes_X\sO_{X_j}.
  \end{equation}
  Comparing \autoref{eq:37} and \autoref{eq:38} shows that $\sfh^{-i}(\dcx{X/B})$
  commutes with base change to $B_q$ for every $q\in\bN$ and then
  $\sfh^{-i}(\dcx{X/B})$ commutes with arbitrary base change by
  \autoref{lem:D-duality} and \cite[1.9]{MR555258}.
  Using that $M_q=\sfh^{-i}(\dcx{X_q/B_q})$ is flat over $B_q$ for every $q\in\bN$,
  it follows that $\sfh^{-i}(\dcx{X/B})$ is flat over $B$ by
  \cite[\href{http://stacks.math.columbia.edu/tag/0523}{Tag 0523}]{stacks-project}.
\end{proof}

\begin{cor}\label{cor:flatness}
  Let $f:(X,x)\to (B,b)$ be a flat local morphism that is essentially of finite type.
  If $(X_{b},x)$ has \toploccohs over $B$ then $\sfh^{-i}(\dcx{X/B})$ is flat over
  $B$ and commutes with arbitrary base change for each $i\in\bZ$.
\end{cor}

\begin{proof}
  By \autoref{thm:key}, \autoref{thm:surjectivity}\autoref{item:1} and
  \autoref{thm:surjectivity}\autoref{item:4} $f$ satisfies the assumptions of
  \autoref{thm:generalized-flat-and-base-change} and hence the statement follows from
  \autoref{thm:generalized-flat-and-base-change}.
\end{proof}

\noin Now we are ready to prove \autoref{thm:main.new}.

\begin{thm}[=~\autoref{thm:main.new}]\label{thm:main-strong}
  Let $f:X\to B$ be a flat morphism of schemes that is essentially of finite type and
  let $b\in B$ 
  such that $X_{b}$ has \toploccohs over $B$. Then there exists an open neighborhood
  $X_b\subset U\subset X$ such that $\sfh^{-i}(\dcx{U/B})$ is flat over $B$ and
  commutes with base change for each $i\in\bZ$.
\end{thm}

\begin{proof}
  Let $x\in X_b$ and temporarily replace $f:X\to B$ with the induced local morphism
  $(X,x)\to (B,b)$.  Then $\sfh^{-i}(\dcx{X/B})$ is flat over $B$ and commutes with
  arbitrary base change by \autoref{cor:flatness}.

  Since localization is an exact functor, we obtain that for the original $f:X\to B$
  and any $x\in X_b$ the localized cohomology sheaves $\sfh^{-i}(\dcx{X/B})_x$ are
  flat over $B$ and commute with base change for each $i\in\bZ$. Both of these
  properties are open on $X$ and hence there is an open neighbourhood of $x$ where
  they hold. The union of these neighbourhoods for all $x\in X_b$ provide an open
  neighbourhood of $X_b$ where these properties hold.
\end{proof}

Now \autoref{thm:main-strong} and \autoref{prop:loc-coh-surj-1} implies
\autoref{cor:main-db-implies-slc}\autoref{item:7} and \autoref{thm:main-strong} and
\autoref{prop:F-anti-nilp-has-liftable-cohs} implies
\autoref{cor:main-db-implies-slc}\autoref{item:8}.

\section{Du\thinspace Bois singularities}\label{sec:local-cohomology-db}

\noin %
In characteristic 0, the optimal setting for deformation invariance of cohomology
seems to be the class of \DB singularities, introduced by Steenbrink
\cite{Steenbrink83}.  For a proper complex variety with \DB singularities the natural
morphism
\begin{equation*}
  \label{eq:star}
  \xymatrix{%
    H^i(X,\bC_X)\ar@{->>}[r] & H^i(X,\sO_X)
  }
  \tag{$\star$}
\end{equation*}
is surjective, and one should think of \DB singularities as the largest class for
which this holds, cf.~\cite{Kovacs11d}. This surjectivity enables one to use
topological arguments to control the sheaf cohomology groups $H^i(X,\sO_X) $ in flat
families as in \cite{MR0376678}.



The proof of \autoref{cor:cm-defo-implies-cm} for projective morphisms in
\cite{MR2629988} very much relied on global duality, hence properness. Our first hope
was that one can localize the proofs by replacing \autoref{eq:star} with the
analogous map between local cohomology groups
\begin{equation*}
  \xymatrix{%
    H^i_x(X,\bC_X)\ar@{->>}[r] & H^i_x(X,\sO_X).
  }
\end{equation*}
However, this turned out to be too simplistic, one needs to consider instead the map
\begin{equation*}
  \xymatrix{%
    H^i_x(X,\bC_X)\ar@{->>}[r] & \bH^i_x(X,\dbcx X),
  }
\end{equation*}
where $\dbcx X$ denotes the $0^\text{th}$ associated graded Du~Bois complex of $X$.
For the construction of the Du~Bois complex see \cite{DuBois81,GNPP88} and for its
relevance to higher dimensional geometry see \cite[\S 6]{SingBook}.  The surjectivity
in \autoref{eq:star} seems simple, but it is a key element of Kodaira type vanishing
theorems \cite{Kollar87b}, \cite[\S 12]{Kollar95s},\cite{Kovacs00c},\cite{MR2646306}
and leads to various results on deformations of \DB schemes
\cite{MR0376678,MR2629988,KS13}. Eventually we understood that for our purposes the
key property is \toploccohs.

\begin{thm}\label{DB.toploccohs}\label{prop:loc-coh-surj-1}
  Let $X$ be a scheme, essentially of finite type over a field of characteristic $0$.
  Assume that $X$ is \DB. Then $X$ has \toploccohs.
\end{thm}

For the definition of \DB\sings the reader is referred to \cite[\S 6]{SingBook}.  A
scheme defined over a field of characteristic $0$ is said to have \emph{\DB \sings}
if its base extension to $\bC$ does. In addition to the properties mentioned above
recall that \DB \sings are invariant under small deformation by
\cite[4.1]{Kovacs-Schwede11b}.
  
Let us start the proof of \autoref{DB.toploccohs} by recalling the following
statement.

\begin{thm}[\cite{Kovacs-Schwede11b,KS13,MSS17}]
  \label{thm:key-injectivity}
  Let $X$ be a scheme, essentially of finite type over
  a field of characteristic $0$.  Then the natural morphism
  \[
  \sfh^i(\uldcx X) \into \sfh^i(\dcx X)
  \]
  is injective for every $i\in\bZ$. 
\end{thm}

\begin{rem}
  \autoref{thm:key-injectivity} was first proved in
  \cite[Theorem~3.3]{Kovacs-Schwede11b}. 
  A version for pairs, essentially with the same proof, was given in
  \cite[Theorem~B]{KS13}. Both of these were stated for reduced schemes even though
  the proof does not need that assumption.  This was noticed and carefully confirmed
  in \cite[Lemma~3.2]{MSS17} where the proof is carried out in detail for the
  not-necessarily-reduced case.
\end{rem}

\begin{cor}\label{cor:key-surjectivity}
  Let $X$ be a scheme, essentially of finite type over a field of characteristic $0$
  and $x\in X$ a closed
  point.  Then the natural morphism
  \[
  \xymatrix{%
    H^i_x(\sO_X) \ar@{->>}[r] &  \bH^i_x(\dbcx X)
  }
  \]
  is surjective for each $i\in\bZ$.
\end{cor}

\begin{proof}
  Let $E$ be an injective hull of the residue field $\kappa(x)$ 
  as an $\sO_{X,x}$-module. Then by local duality \cite[Theorem~V.6.2]{RD} there
  exists a commutative diagram where the vertical maps are isomorphisms:
  \begin{equation}
    \label{eq:2}
    \begin{aligned}
      \xymatrix{%
        \myR\Gamma_x(\sO_X)\ar[r] \ar[d]_\simeq &  \myR\Gamma_x(\dbcx X) \ar[d]^\simeq \\
        \myR\!\Hom_{\sO_{X,x}} ( \dcx X, E) \ar[r] & \myR\!\Hom_{\sO_{X,x}} ( \uldcx
        X, E). }
    \end{aligned}
  \end{equation}
  Since $E$ is injective, the functor $\Hom_{\sO_{X,x}} ( \blank, E)$ is exact and
  hence it commutes with taking cohomology. Thus one has that 
  \[
  \sfh^i( \myR\!\Hom_{\sO_{X,x}} ( \dcx X, E) ) \simeq \Hom_{\sO_{X,x}} ( \sfh^i(\dcx
  X), E)
  \]
  and
  \[
  \sfh^i( \myR\!\Hom_{\sO_{X,x}} ( \uldcx X, E) ) \simeq \Hom_{\sO_{X,x}} (
  \sfh^i(\uldcx X), E).
  \]
  It follows that by taking cohomology of the diagram in \autoref{eq:2} one obtains
  for each $i$ the commutative diagram
  \begin{equation}
    \label{eq:3}
    \begin{aligned}
      \xymatrix{%
        H^i_x(\sO_X)\ar[r] \ar[d]_\simeq &  \bH^i_x(\dbcx X) \ar[d]^\simeq \\
        \Hom_{\sO_{X,x}} ( \sfh^i(\dcx X), E) \ar[r] & \Hom_{\sO_{X,x}} (
        \sfh^i(\uldcx X), E). }
    \end{aligned}
  \end{equation}
  Again, since $\Hom_{\sO_{X,x}} ( \blank, E)$ is exact, it follows from
  \autoref{thm:key-injectivity} that the bottom homomorphism is surjective which
  implies the desired statement.
\end{proof}

\begin{rem}
  An important aspect of \autoref{cor:key-surjectivity} is that the local cohomology
  of $\sO_X$ depends on the non-reduced structure, while that of $\dbcx X$ does
  not. Essentially, the left hand side reflects the algebraic structure, while the
  right hand side behaves as if it only depended on the topology (this is not
  entirely true!).

  This behavior allows us to prove \autoref{DB.toploccohs}. The proof is based on the
  interplay between the non-reduced and reduced data. It is similar in spirit to the
  proofs of \cite[Lemme 1]{MR0376678}, \cite[Theorem~5.1]{KS13}, and
  \cite[Lemma~3.3]{MSS17}.
\end{rem}

\begin{proof}[Proof of \autoref{DB.toploccohs}]
  Using \autoref{not:top-loc-cohs} assume that $A=k$ is a field of characteristic $0$
  and that $X=\Spec T$ has Du~Bois singularities.  We need to prove that the induced
  morphism on local cohomology $H^i_\frm(R) \onto H^i_\frm(T)$ is surjective for each
  $i$.
  Consider the following diagram:
  \[
  \xymatrix@C=3em@R=3em{%
    H^i_\frm(R) \ar@{->>}[d]_\xi \ar[rr]^-\chi 
    & 
    & H^i_\frm(T) \ar[d]_\simeq^\vartheta \\
    H^i_\frm(\dbcx R) \ar[rr]^-\simeq_-\zeta & & H^i_\frm(\dbcx {T}) }
  \]
  Using the notation of the diagram, one has that $\xi$ is surjective by
  \autoref{cor:key-surjectivity}, $\zeta$ is an isomorphism because
  $\dbcx R\simeq \dbcx {T}$, and $\vartheta$ is an isomorphism, because $\Spec T$ has
  \DB \sings. It follows then that
  $\chi$ is surjective.
\end{proof}





\section{\Fp \sings}\label{sec:fp-sings}

\noin %
There is an intriguing correspondence between singularities of the minimal model
program in characteristic $0$ and singularities defined by the action of the
Frobenius morphism in positive characteristic. For more on this correspondence the
reader may consult \cite[App.~C]{MR2932591} or \cite[\S8.4]{SingBook}. Our goal here
is to show that $F$-pure, or more generally $F$-anti-nilpotent singularities have
\toploccohs over their ground field.

\begin{defini}
  Let $(R,\frm)$ be a noetherian local ring of characteristic $p>0$ with the
  Frobenius endomorphism $F:R\to R$; $x\mapsto x^p$. 

  Recall that a homomorphism of $R$-modules $M\to M'$ is called \emph{pure} if for
  every $R$-module $N$ the induced homomorphism $M\otimes_RN\to M'\otimes_RN$ is
  injective. $R$ is called \emph{$F$-pure} if the Frobenius endomorphism is pure. $R$
  is called \emph{$F$-finite} if $R$ is a finitely generated $R$-module via the
  Frobenius endomorphism $F$. For instance, if $R$ is essentially of finite type over
  a field, then it is $F$-finite.  Further note that if $R$ is $F$-finite or complete
  then it is $F$-pure if and only if the Frobenius endomorphism $F:R\to R$ has a left
  inverse \cite[5.3]{MR0417172}.

  $R$ is called \emph{$F$-injective} if the induced Frobenius action on $H^i_\frm(R)$
  is injective for all $i\in \bN$. This holds for example if $R$ is $F$-pure by
  \cite[2.2]{MR0417172} and if $R$ is Gorenstein then it is $F$-pure if and only if
  it is $F$-injective \cite[3.3]{MR701505}.  

  A strengthening of the notion of $F$-injective was recently introduced in
  \cite{MR2460693}: Consider the induced Frobenius action
  $F:H^i_\frm(R)\to H^i_\frm(R)$. A submodule $M\subseteq H^i_\frm(R)$ is called
  \emph{$F$-stable} if $F(M)\subseteq M$ and $R$ is called \emph{$F$-anti-nilpotent}
  if for any $F$-stable submodule $M\subseteq H^i_\frm(R)$, the induced Frobenius
  action on the quotient $H^i_\frm(R)/M$ is injective.  If $R$ is $F$-anti-nilpotent,
  then it is $F$-injective, since $\{0\}\subseteq H^i_\frm(R)$ is an $F$-stable
  submodule. Furthermore, if $R$ is $F$-pure, then it is $F$-anti-nilpotent by
  \cite[3.8]{MR3271179}. So we have the following implications:
  \begin{equation}
    \label{eq:13}
    \xymatrix{%
      \text{$F$-pure } \ar@{=>}[r] &  \text{ $F$-anti-nilpotent } \ar@{=>}[r] & 
      \text{ $F$-injective}. 
    }
  \end{equation}
  Let $(X,x)$ be a local scheme. Then we say that $X$ has \emph{$F$-pure},
  resp.~\emph{$F$-anti-nilpotent}, resp.~\emph{$F$-injective} singularities if the
  local ring $\sO_{X,x}$ has the corresponding property. An arbitrary scheme $X$ of
  equicharacteristic $p>0$ has \emph{$F$-pure}, resp.~\emph{$F$-anti-nilpotent},
  resp.~\emph{$F$-injective} singularities if the local scheme $(X,x)$ has the
  corresponding property for each $x\in X$.
\end{defini}

These singularities are related to the singularities of the minimal model
program. Normal $\bQ$-Gorenstein $F$-pure singularities are log canonical by
\cite{MR1874118} and it is conjectured that in some form the converse also
holds. Similarly, $F$-injective singularities correspond to \DB \sings: If $X$ is
essentially of finite type over a field of characteristic $0$ and its reduction mod
$p$ is $F$-injective for infinitely many $p$'s, then $X$ has \DB \sings by
\cite{MR2503989} and the converse of this is also conjectured to hold. So the
(outside) implication in \autoref{eq:13} is analogous to that log canonical \sings
are \DB \cite{MR2629988}.

Curiously, we have this additional notion, $F$-anti-nilpotent, in between the more
familiar $F$-pure and $F$-injective notions. It turns out that $F$-anti-nilpotent,
and hence $F$-pure, rings have \toploccohs, but $F$-injective in general do not. This
suggests that possibly $F$-anti-nilpotent is a better analog of \DB \sings in
positive characteristic than $F$-injective. Of course, this is far from conclusive
evidence, and this issue will not be settled here.

These singularities, defined by the action of Frobenius, have been studied
extensively through their local cohomology. So it is no surprise that the fact that
$F$-anti-nilpotent singularities have \toploccohs is a relatively simple consequence of
known results. The following statement is essentially proved in
\cite[Remark~3.4]{MSS17}, although their statement is slightly different, so we
include a proof for completeness.

\begin{prop}[(Ma-Schwede-Shimomoto)]\label{prop:F-anti-nilp-has-liftable-cohs}
  $F$-anti-nilpotent singularities have \toploccohs over their ground field.
\end{prop}

\begin{proof}
  (Following the argument in \cite[Remark~3.4]{MSS17}).
  Using \autoref{not:top-loc-cohs} assume that $A=k$ is a field of characteristic
  $p>0$ and that $(R,\frm)$ has \Fan \sings.  We need to prove that the induced
  morphism on local cohomology $H^i_\frm(R) \onto H^i_\frm(T)$ is surjective for each
  $i$ (cf.~\autoref{def:liftable-cohom}).

  Since the statement is about local cohomology we may assume that $R$ is complete
  and hence $R\simeq R'/J$ where $R'$ is a complete regular local ring and
  $J\subseteq R'$ is an ideal.  Note that denoting the pre-image of $I\subset R$ in
  $R'$ by $I'$, we have that $T\simeq R'/I'$ is also a quotient of $R'$.

  Let $M\leteq \im\left[H^i_\frm(R)\to H^i_\frm(T)\right]$, which is an $F$-stable
  submodule of $H^i_\frm(T)$. Then $M$ contains $F^e(H^i_\frm(T))$ for some $e>0$ by
  \cite[Lemma~2.2]{MR2197409} and hence the Frobenius action on $H^i_\frm(T)/M$ is
  nilpotent.  In particular it is injective only if this quotient is $0$. Therefore,
  if $R$ is $F$-anti-nilpotent, then $H^i_\frm(R)\onto H^i_\frm(T)$ is surjective as
  desired.
\end{proof}

\begin{cor}[(Ma)]\label{cor:F-pure-has-liftable-cohs}
  $F$-pure singularities have \toploccohs over their ground field.
\end{cor}

\begin{proof}
  $F$-pure \sings are $F$-anti-nilpotent by \cite[3.8]{MR3271179}, so this is a
  direct consequence of \autoref{prop:F-anti-nilp-has-liftable-cohs}.
\end{proof}

\section{Degenerations of \CM \sings with \toploccohs}\label{sec:degenerations-cm-db}

\begin{prop}\label{prop:S_n-via-hi}
  Let $Z$ be a scheme that admits a dualizing complex $\dcx{Z}$, $z\in Z$ a
  (not-necessarily-closed) point, and $n\in\bN$. Then $Z$ is $S_n$ at $z$ if and only
  if for all $i\in\bZ$,
  \begin{equation}
    \label{eq:27}
    \sfh^{-i}(\dcx Z)_z = 0 \text{ for $i<\min(n,\dim_z Z)+\dim z$}.
  \end{equation}
  In particular, if $Z$ is equidimensional, then $Z$ is $S_n$ if and only if for all
  $i\in\bZ$, $i<\dim Z$,
  \begin{equation}
    \label{eq:29}
    \dim\supp\sfh^{-i}(\dcx Z) 
    \leq i-n.
  \end{equation}
  (In this statement we take $\dim \emptyset=-\infty$).
\end{prop}

\begin{proof}
  Since $\sfh^{-i}(\dcx Z)=\sExt^{-i}_Z(\sO_Z,\dcx Z)$, \autoref{eq:27} follows
  directly from \cite[Prop~3.2]{MR2918171}. 
  
  Next, let $i\in\bZ$, $i<\dim Z$ be such that $\sfh^{-i}(\dcx Z)\neq 0$ and let
  $z\in\supp\sfh^{-i}(\dcx Z)$ be a general point such that
  $\dim z=\dim\supp\sfh^{-i}(\dcx Z)$.
  If $Z$ is $S_n$, then $i\geq \min(n,\dim_z Z)+\dim z$ by \autoref{eq:27} and hence,
  since $i<\dim Z$, we must have $\min(n,\dim_z Z)=n$ (this is where $Z$ being
  equidimensional is used), so indeed $\dim\supp\sfh^{-i}(\dcx Z) \leq i-n$.

  In order to prove the other implication let $i\in\bZ$, $i<\dim Z$ be again such
  that $\sfh^{-i}(\dcx Z)\neq 0$, but now choose an arbitrary point
  $z\in\supp\sfh^{-i}(\dcx Z)$. In this case we only have that
  $\dim z\leq \dim\supp\sfh^{-i}(\dcx Z)$, but this will be enough.  If
  $\dim\supp\sfh^{-i}(\dcx Z) \leq i-n<\dim Z -n$, then
  $n< \dim Z - \dim\supp\sfh^{-i}(\dcx Z)\leq \dim Z -\dim z =\dim_zZ$, i.e.,
  $\min(n,\dim_z Z)=n$.  It also follows that
  $i\geq n+ \dim\supp\sfh^{-i}(\dcx Z)\geq\min(n,\dim_z Z) +\dim z$, and hence $Z$ is
  $S_n$ at $z$ by \autoref{eq:27}. We obtain that for all $i<\dim Z$,
  $\supp\sfh^{-i}(\dcx Z)$ is contained in the $S_n$-locus of $Z$.  However, $Z$ is
  \CM and hence $S_n$ at every point in
  $Z\setminus\bigcup_{i<\dim Z}\supp\sfh^{-i}(\dcx Z)$, which proves \autoref{eq:29}.
\end{proof}

\begin{cor}
  Let $Z$ be an equidimensional scheme that admits a dualizing complex $\dcx{Z}$. If
  $Z$ is $S_n$ for some $n\in\bN$, then $\sfh^{-i}(\dcx Z)=0$ for $i<n$. 
\end{cor}

\begin{thm}\label{cm-is-a-closed-prop}
  Let $f:X\to B$ be a flat morphism with equidimensional fibers that is locally
  embeddable into a smooth morphism.
  Assume that there exists a $b_0\in B$ such that $X_{b_0}$ has \toploccohs over $B$.
  If $X_{b_0}$ is not $S_n$ then there exists an open subset $b_0\in V\subseteq B$
  such that $X_b$ is not $S_n$ for each $b\in V$.
\end{thm}

\begin{proof}
  By \autoref{thm:main-strong} there exists an open neighborhood
  $X_b\subset U\subset X$ such that $\sfh^{-i}(\dcx{U/B})$ is flat over $B$ and
  commutes with base change for each $i\in\bZ$.  Then
  $\dim \supp\sfh^{-i}(\dcx{U_b})$ is a locally constant function on the set
  $\{b\in B\skvert \sfh^{-i}(\dcx{X_b})\neq 0\}$, so the claim follows from
  \autoref{eq:29}.
\end{proof}

\subsection{Deformations of local schemes}

\begin{defini}
  Let $(A,\frm_1,\dots,\frm_r)$ be a semi-local ring. Then $(X,x_1,\dots,x_r)$ is
  called a \emph{semi-local scheme} where $X=\Spec A$ and
  $x_1=\frm_1,\dots,x_r=\frm_r \in X$. If $r=1$ and $A$ is a local ring then
  $(X,x_1)$ is a \emph{local scheme}.

  A \emph{family of semi-local schemes} consists of a pair $(\mcX,\mcx)$ where
  $\mcx\subseteq \mcX$ is a closed subscheme and a flat morphism $f:\mcX\to B$ that
  is essentially of finite type such that $f\resto{\mcx}:\mcx\to B$ is a dominant
  finite morphism and for any $b\in B$, $(\mcX_b,\red(\mcx_b))$ is an equidimensional
  semi-local scheme. By a slight abuse of notation this family of semi-local schemes
  will be denoted by $f:(\mcX,\mcx)\to B$.

  Let $\sfP$ be a local property of a scheme such as being \DB, 
  \Fp, $F$-anti-nilpotent, 
  $S_n$, or \CM. We will say that a semi-local scheme $(X,x_1,\dots,x_r)$ has
  property $\sfP$ if $X$ is $\sfP$ at $x_1,\dots,x_r$. In particular, we will say
  that $(X,x_1,\dots,x_r)$ is a \DB semi-local scheme, etc.
  Similarly for ``local scheme'' in place of ``semi-local scheme''.
  %
\end{defini}


\begin{thm}\label{thm:cm-defo-of-topcoh}\label{thm:cm-defo-of-db}
  Let $f:(\mcX,\mcx)\to B$ be a family of semi-local schemes. Assume that
  \begin{enumerate}
  \item\label{item:15} $B$ is irreducible,
  \item\label{item:16} the fibers of $f$ are equidimensional,
  \item\label{item:17} there exists a $b_0\in B$ such that $\mcX_{b_0}$ has
    \toploccohs over $B$, and
  \item\label{item:18} there exists a $b_1\in B$ and an $n\in\bN$ such that
    $\mcX_{b_1}$ is $S_n$ (resp.~\CM).
  \end{enumerate}
  Then there exists an open set $b_0\in U\subseteq B$ such that $\mcX_{b}$ is $S_n$
  (resp.~\CM) for each $b\in U$. In particular, $\mcX_{b_0}$ is $S_n$ (resp.~\CM).
\end{thm}

\begin{proof}
  It is enough to prove the statement for the $S_n$ property.  Since $f\resto \mcx$
  is proper, 
  the set $U:=\{b\in B \skvert \mcX_b \text{ is } S_n \}$ is open in $B$ by
  \cite[12.1.6]{EGAIV3}.  By \autoref{item:18} it is non-empty and hence it is dense
  in $B$.  Then it must contain $b_0$ by \autoref{cm-is-a-closed-prop}, which proves
  the statement.
\end{proof}

\begin{rem}
  If $\mcX_{b_0}$ has \DB \sings then we may even choose $U$ such that $\mcX_b$ is
  $S_n$ and has \DB \sings for all $b\in U$ by \cite[4.1]{Kovacs-Schwede11b}. As we
  mentioned earlier, it is not known whether small deformations of \Fan \sings remain
  \Fan. It is also an interesting question whether the condition of having
  \toploccohs is invariant under small deformations.
\end{rem}

It follows that \autoref{thm:cm-defo-of-topcoh} applies to families of semi-local
schemes with \DB or $F$-anti-nilpotent \sings by \autoref{prop:loc-coh-surj-1} and
\autoref{prop:F-anti-nilp-has-liftable-cohs}
(cf.~\autoref{rem:top-loc-cohs-is-hereditary}).

%
%

\noin As a simple consequence we obtain a generalization of
\autoref{cor:cone-over-abelian}.

\begin{cor}\label{cor:cone-over-abelian-s3}
  Let $Z$ be a normal projective variety over a field $k$ such that $K_Z$ is
  $\bQ$-Cartier and numerically equivalent to $0$. Let $\sL$ be an ample line bundle
  on $Z$ and $X=C_a(Z,\sL)$ the affine cone over $Z$ with conormal bundle $\sL$. If
  $X$ has \toploccohs over $k$ and admits an $S_n$ deformation for some $n\in\bN$,
  then $H^i(Z,\sO_Z)=0$ for $0<i<n-1$.  In particular, a cone over an abelian variety
  (ordinary, if $\kar k>0$) of dimension at least $2$ does not admit an $S_3$
  deformation.
\end{cor}

\begin{proof}
  If $X$ admits an $S_n$ deformation for some $n\in\bN$, then $X$ itself is $S_n$ by
  \autoref{thm:cm-defo-of-topcoh} and the first statement follows from
  \autoref{item:24}. Then the second statement follows from
  \autoref{cor:cone-over-abelian} and \autoref{DB.toploccohs} in characteristic $0$
  and from \autoref{exmp:sing-of-cones-ii} and \autoref{cor:F-pure-has-liftable-cohs}
  in positive characteristic.
\end{proof}

\begin{rem}
  As noted in the introduction, it is proved in \cite{MR522037} that the projective
  cone over an abelian variety of dimension at least $2$ over $\bC$ is not smoothable
  inside the ambient projective space. This is a strict special case of
  \autoref{cor:cone-over-abelian-s3}.  In general, it is possible that a
  non-smoothable projective variety is locally smoothable. An example of that is
  given in \cite[2.18]{CoughlanTaro16}. Furthermore,
  \autoref{cor:cone-over-abelian-s3} is also valid in positive characteristic.
\end{rem}


\def\cprime{$'$} \def\polhk#1{\setbox0=\hbox{#1}{\ooalign{\hidewidth
  \lower1.5ex\hbox{`}\hidewidth\crcr\unhbox0}}} \def\cprime{$'$}
  \def\cprime{$'$} \def\cprime{$'$} \def\cprime{$'$}
  \def\polhk#1{\setbox0=\hbox{#1}{\ooalign{\hidewidth
  \lower1.5ex\hbox{`}\hidewidth\crcr\unhbox0}}} \def\cdprime{$''$}
  \def\cprime{$'$} \def\cprime{$'$} \def\cprime{$'$} \def\cprime{$'$}
  \def\cprime{$'$}
\providecommand{\bysame}{\leavevmode\hbox to3em{\hrulefill}\thinspace}
\providecommand{\MR}{\relax\ifhmode\unskip\space\fi MR}
\providecommand{\MRhref}[2]{%
  \href{http://www.ams.org/mathscinet-getitem?mr=#1}{#2}
}
\providecommand{\href}[2]{#2}

\end{document}